\documentclass[]{amsart}

\usepackage{graphicx,color,hyperref}

\usepackage[initials,nobysame]{amsrefs}
\usepackage{amssymb,amsmath}
\usepackage{subcaption}
\usepackage{comment, multirow}
\usepackage{caption}
\usepackage{overpic, tikz}
\usepackage{booktabs}

\numberwithin{equation}{section}

\newtheorem{theorem}{Theorem}[section]
\newtheorem{proposition}[theorem]{Proposition}
\newtheorem{corollary}[theorem]{Corollary}
\newtheorem{lemma}[theorem]{Lemma}

\theoremstyle{definition}
\newtheorem{definition}[theorem]{Definition}

\theoremstyle{remark}
\newtheorem{remark}[theorem]{Remark}

\DeclareMathOperator{\sech}{sech}
\DeclareMathOperator{\cn}{cn}

\DeclareMathOperator{\am}{am}

\DeclareMathOperator{\spt}{spt}
\DeclareMathOperator{\Span}{span}
\DeclareMathOperator{\so}{SO}

\newcommand{\eps}{\varepsilon}

\newcommand{\R}{\mathbf{R}}
\newcommand{\Z}{\mathbf{Z}}
\newcommand{\C}{\mathcal{C}}
\newcommand{\A}{\mathcal{A}}

\newcommand{\N}{\mathbf{N}}

\newcommand{\bS}{\mathbf{S}}
\newcommand{\B}{\mathcal{B}}
\newcommand{\Ll}{\mathcal{L}}

\begin{document}

\title{Regularity and structure of non-planar $p$-elasticae}

\author[F.~Gruen]{Florian Gruen}
\address[F.~Gruen]{Department of Mathematics, Graduate School of Science, Kyoto University, Kitashirakawa Oiwake-cho, Sakyo-ku, Kyoto 606-8502, Japan}
\email{gruen.florian.32r@st.kyoto-u.ac.jp}

\author[T.~Miura]{Tatsuya Miura}
\address[T.~Miura]{Department of Mathematics, Graduate School of Science, Kyoto University, Kitashirakawa Oiwake-cho, Sakyo-ku, Kyoto 606-8502, Japan}
\email{tatsuya.miura@math.kyoto-u.ac.jp}


\begin{abstract}
    We prove regularity and structure results for $p$-elasticae in $\R^n$, with arbitrary $p\in (1,\infty)$ and $n\geq2$. 
    Planar $p$-elasticae are already classified and known to lose regularity. 
    In this paper, we show that every non-planar $p$-elastica is analytic and three-dimensional, with the only exception of flat-core solutions of arbitrary dimensions.
    Subsequently, we classify pinned $p$-elasticae in $\R^n$ and, as an application, establish a Li--Yau type inequality for the $p$-bending energy of closed curves in $\R^n$. This extends previous works for $p=2$ and $n\geq2$ as well as for $p\in (1,\infty)$ and $n=2$.
\end{abstract}

\maketitle

\section{Introduction}

The elastica is a classical problem in geometric analysis, dating back to Daniel Bernoulli and Leonhard Euler. 
Originally used to model the bending of an elastic wire in physical space, it has been studied thoroughly from purely mathematical (e.g.\ \cites{langer-singer_classification,langersinger_minimax,singer_lecturenotes, miura_phase_transitions}) and applied (e.g.\ \cites{CANHAM197061_bloodcell,Helfrich+1973+693+703}) perspectives, see also the surveys \cites{Levien:EECS-2008-103, elastica_influence, elastica_survey} and the references therein. 
Classically, for an immersed curve $\gamma$, the (squared) bending energy, i.e.\ the integral of the squared curvature is considered. 
Recently, a generalization, the so-called $p$-\emph{bending energy} has attracted much attention in the mathematical community \cites{acerbi, ambrosio_image, ARROYO2003339, p-elastica_in_sphere, kawohl1, masnou_levellines,p-obstacle, miuraclassification, Okabe_p_flow, Okabe_p_flow2, Pozzetta1, p_network_flow, Japan_p}.

Let $p\in(1,\infty)$ and $\gamma \in W^{2,p}(a,b;\R^n)$ be an immersed curve.
Then its $p$-bending energy is defined as
\begin{equation}
    \B_p[\gamma] := \int_\gamma |\kappa|^p ds,
\end{equation}
where $\kappa$ denotes the curvature vector of $\gamma$.
Since the $p$-bending energy is invariant under reparameterization and Euclidean isometries, we may assume $\gamma$ to be arclength parameterized, unless specified otherwise, and hence $\kappa(s)=\gamma''(s)$. 

Reminiscent of Euler's classical elastica problem, we additionally restrict the curve $\gamma$ to have fixed length $\mathcal{L}[\gamma] = \int_\gamma ds=L$. 
This constraint is present in form of a Lagrange multiplier $\lambda \in \R$.

\begin{definition}[$p$-elastica] 
\label{def: p elastica}
Let $p\in (1,\infty)$ and $\gamma \in W^{2,p}(0,L;\R^n)$ be an 
arclength parameterized curve.
Then $\gamma$ is called a \emph{$p$-elastica} if there exists $\lambda \in \R$ such that $\gamma$ is a critical point of the map $\gamma \mapsto \B_p[\gamma] + \lambda \Ll[\gamma]$, i.e.\ 
\begin{equation}
\label{eq: first variation frechet}
  \forall \eta \in C^\infty_c(0,L;\R^n): \qquad  \frac{d}{d\eps}\bigg|_{\eps=0} \left( \B_p[\gamma + \eps \eta] + \lambda \Ll[\gamma+\eps \eta] \right)  = 0.
\end{equation}
\end{definition}

We note that Definition~\ref{def: p elastica} is equivalent to the usual definition in terms of immersed curves, cf.\ \cite[Appendix~ A]{miuraclassification}.
We say that a $p$-elastica $\gamma$ is \emph{$d$-dimensional} if its image $\gamma([0,L])$ is contained in a $d$-dimensional (but not a $(d-1)$-dimensional) affine subspace of $\R^n$.
We also say \emph{planar} (resp.\ \emph{spatial}) in the same sense as two dimensional (resp.\ three dimensional).

The main aim of our work is to enhance the mathematical understanding of the $p$-bending energy as well as the regularity and structural properties of $p$-elasticae. 
In addition to the classical case $p=2$ (elastica), recent work \cites{infinite_elastica_classification, infinity_elastica_manifold, kaya2024curvesminimaxcurvature} analyzes an $L^\infty$-bending energy, which is closely connected to other geometric optimization problems such as the Markov--Dubins problem, see \cite{infinite_elastica_classification} and the references therein.
We note that $\infty$-elasticae in \cite{infinite_elastica_classification} are obtained by approximating the $L^\infty$-norm by the $L^p$-norm, passing to the limit as $p\to \infty$ in the corresponding Euler--Lagrange equations and then studying the solutions of the limiting ODE system. 
The $p$-bending energy (at least for $p>2$) can thus be seen as an intermediate setting between the $L^2$ and the $L^\infty$ case.
One of our main results, Theorem~\ref{thm:p>2 case}, exhibits a strong structural resemblance to the classification result \cite[Theorem~4]{infinite_elastica_classification} for $\infty$-elasticae; with this paper we also hope to open up new insights into the relationship between $p$-elasticae and $\infty$-elasticae.
Additionally, $p$-elasticae can be interpreted as stationary solutions of the associated $p$-elastic gradient flows, higher order parabolic evolution equations, see Section~\ref{section: p-elastic flows}.

For the case $p=2$, every elastica is analytic and (at most) three dimensional; also, classification, stability and explicit parameterizations have been known since the 80's, thanks to the works of Langer and Singer \cites{langer-singer_classification,singer_lecturenotes,langersinger_minimax}, see also \cite{elastica_survey}. 
The case $p\neq2$ is more delicate; we need to differentiate between the singular ($1<p<2$) and degenerate ($p>2$) setting. 
This terminology originates from the leading order term in the corresponding Euler--Lagrange equation \eqref{eq:EL for k and tau}. 
In the planar case, classification and explicit parameterizations (involving newly defined $p$-elliptic functions) have been recently obtained in \cite{miuraclassification}.
In addition, optimal regularity is determined depending on $p\in(1,\infty)$ \cite[Theorem 1.8 and Theorem 1.9]{miuraclassification}, which, in particular, implies that the regularity of planar $p$-elasticae is lost for all but countably many $p$'s.

\subsection{Main results}

As our main results, we obtain optimal regularity and structural properties for $p$-elasticae with $p \in (1,\infty)$ in $\R^n$. 
First, in the singular case ($1<p<2$), the regularity remains consistent with the case $p=2$.
Notably, and in stark contrast to the planar case, the loss of regularity is entirely absent.

\begin{theorem}
\label{thm:p<=2 case}
    Let $p\in(1,2]$ and $n\geq3$.
    Then any non-planar $p$-elastica in $\R^n$ is analytic and three dimensional.
\end{theorem}

In the degenerate case ($p>2$), the situation changes and so-called \emph{flat-core solutions} (first introduced in~\cite{watanabe_flatcore}, see Figure~\ref{fig: flat-core example1} and Definition~\ref{def:flatcore}) emerge. 
These special solutions are not restricted to a three dimensional subspace, but are potentially ``proper'' $n$-dimensional curves.
We have the following dichotomy:

\begin{figure}[htbp]
    \centering
    \includegraphics[width=0.8\linewidth, trim = 300 240 300 200, clip]{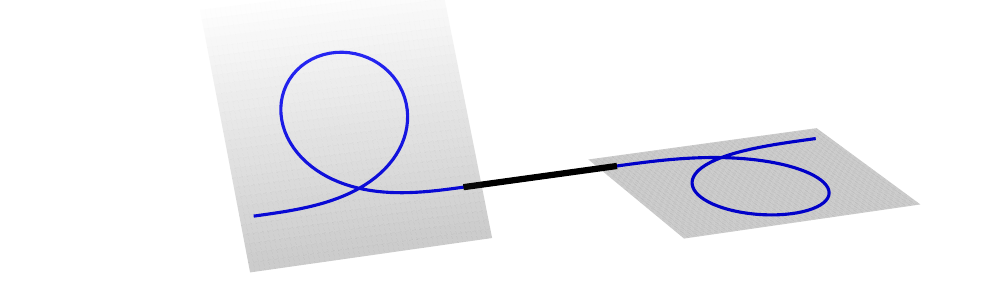}
    \caption{A three dimensional flat-core $p$-elastica ($p=4$): two loops connected by a straight line part.}
    \label{fig: flat-core example1}
\end{figure}

\begin{theorem}
\label{thm:p>2 case}
    Let $p\in(2,\infty)$ and $n\geq3$.
    Then any non-planar $p$-elastica in $\R^n$ is either
    \begin{enumerate}
        \item an analytic three dimensional $p$-elastica, or 
        \item a $d$-dimensional flat-core $p$-elastica for some $d\in\{3,\dots,n\}$. For $M_p:=\lceil \tfrac{2}{p-2} \rceil$, it is of class $C^{M_p+1}$ but not of class $C^{M_p+2}$. 
    \end{enumerate}
\end{theorem}

The main difficulty in the proof lies in analyzing the $p$-elastica $\gamma$ at points of vanishing curvature (joints). 
We first prove that in this case $\gamma$ is partially planar. However, regularity across the joints is lost and in general the strong Euler--Lagrange equation may not hold globally.
We overcome this issue by constructing suitable test functions for the global weak formulation to extract sufficient geometric conditions at the joints.

Finally, we obtain a nonlinear ODE for the curvature $k$ and torsion $\tau$ in the generic, non-planar and analytic case.

\begin{theorem}
\label{thm:EL for k and tau}
    Let $p\in (1,\infty)$ and $\gamma:[0,L]\to\R^n$ be an analytic non-planar $p$-elastica.
    Then $\gamma$ is a spatial curve with $k,|\tau|>0$ on $[0,L]$, satisfying 
    \begin{equation}
    \label{eq:EL for k and tau}
    \begin{split}
        p (k^{p-1})'' + (p-1) k^{p+1} -p k^{p-1} \tau^2 -\lambda k &= 0, \\
         k^{2p-2} \tau &= C,
    \end{split}
    \end{equation}
for some real constant $C\neq 0$.
\end{theorem}
In the classical $p=2$ case, \cite{langer-singer_classification}, the curvature and torsion are explicitly expressed using classical Jacobi elliptic functions as a result of the polynomial structure of \eqref{eq:EL for k and tau}.
However, for the general case $p \in (1,\infty)$, explicit solutions for $k$ and $\tau$ and their stability properties remain unknown.

\subsection{Pinned boundary value problem}

Recent work \cite{miura_pinned_p} classifies pinned (zeroth-order boundary condition) planar $p$-elasticae. 
Similarly, for $p\in (1,\infty)$, $L>0$ and $P_0,P_1 \in \R^n$ with $|P_0-P_1|<L$, let $\A_{P_0,P_1,L}$ be the set of admissible curves in $\R^n$,
\begin{equation*}
    \A_{P_0,P_1,L}:=\{\gamma \in W^{2,p}(0,L;\R^n): |\gamma'|\equiv 1, \ \gamma(0)=P_0, \ \gamma(L)=P_1\}.
\end{equation*}
A pinned $p$-elastica is defined as a critical point of $\B_p$ within this admissible set.
Thanks to the multiplier method and reparameterization invariance (see \cite{miura_pinned_p}), the problem can equivalently be reformulated as follows:

\begin{definition}
\label{def: pinned p elastica}
    Given $p\in (1,\infty)$, we say that $\gamma \in \A_{P_0,P_1,L}$ is a \emph{pinned $p$-elastica} if there exists $\lambda \in \R$ such that 
    \begin{equation*}
        \forall \eta \in C_0^\infty(0,L;\R^n): \frac{d}{d\eps} \bigg|_{\eps=0} (\B_p[\gamma + \eps \eta] + \lambda \Ll[\gamma + \eps \eta]) =0.
    \end{equation*}
\end{definition}

We extend the classification in \cite{miura_pinned_p} from $\R^2$ to $\R^n$ by employing our regularity results to derive natural boundary conditions and essentially reduce the problem to the planar case.

\begin{theorem}
\label{thm: pinned classification}
Let $p\in (1,\infty)$, $L>0$ and $P_0,P_1 \in \R^n$ with $|P_0-P_1|<L$, and $\gamma \in \A_{P_0,P_1,L}$ be a pinned $p$-elastica. Then $k(0)=k(L)=0$, and we have the following classification:
\begin{enumerate}
    \item If  $p \leq 2$, or if $p > 2$ and $|P_0-P_1| < \frac{L}{p-1}$, then $\gamma$ is planar.
    \item If otherwise, then $\gamma$ is either planar, or a non-planar flat-core $p$-elastica whose aligned representative $\hat{\gamma}$ (see Definition~\ref{def: flatcore representation}) is a planar pinned $p$-elastica.
\end{enumerate}
\end{theorem}

Thanks to the essentially planar nature, we obtain the following corollary.
\begin{corollary}
\label{cor: unique pinned minimizer}
    Let $p\in (1,\infty)$, $L>0$ and $P_0,P_1 \in \R^n$ with $|P_0-P_1|<L$. Then there exists a unique (up to isometry) global minimizer $\gamma$ of $\B_p$ in $\A_{P_0,P_1,L}$, which is planar (and hence a convex arc given in \cite[Theorem~1.4]{miura_pinned_p}).
\end{corollary}

\subsection{Applications}

Langer and Singer showed in \cite{langersinger_minimax} that the only spatially stable closed elastica in $\R^3$ is the one fold circle, while other closed elasticae (e.g.\ the well-known \emph{figure-eight} curve \cite{langer-singer_classification}) are spatially unstable. 
Whether a similar result holds for the case $p\neq2$ is currently open. 
Further ``physically stable" configurations for $p=2$ are studied by considering self-intersections.
Some of such configurations are given by optimal curves for the Li--Yau type inequality in \cite{miura_LiYau} (see also \cite{elastica_survey}).
The Li–Yau type inequality establishes a sharp relationship between the bending energy and the multiplicity at a point.

First, for $p\in (1,\infty)$, we define the \emph{normalized $p$-bending energy} as 
\begin{equation*}
    \bar{\B}_p[\gamma] := \Ll [\gamma]^{p-1} \B_p[\gamma],
\end{equation*}
which is invariant under dilations, $\gamma \mapsto \Lambda \gamma$, for positive constants $\Lambda$. 
Furthermore, we say that $\gamma$ has a point $P$ of multiplicity $m$, if $\gamma^{-1}(P)$ has $m$ distinct points.

In \cite{miura_LiYau}, it is shown that there exists a universal constant $\varpi_2^*\approx28.109$ (calculated by \eqref{eq: varpi constant} with $p=2$) such that for an immersed closed curve $\gamma \in W^{2,2} (\R / \Z ; \R^n)$ with a point of multiplicity $m$,
\begin{equation*}
    \bar{\B}_2 [\gamma] \geq \varpi_2^* m^2, 
\end{equation*}
where equality holds if and only if $n\geq 3$ or $m$ is even, and $\gamma$ an \emph{$m$-leafed elastica} (see Definition~\ref{def: leafed elastica}). 

In \cite{miura_pinned_p}, the above Li--Yau type inequality is extended from $p=2$ to $p\in(1,\infty)$, instead restricting to the case $\R^2$.
This restriction is necessary since the proof depends on regularity and classification results only available for planar $p$-elasticae. 

In this work, based on our results in $\R^n$, we extend the above Li--Yau type inequalities, namely to general dimensions $n\in\N_{\geq 2}$ and general exponents $p\in(1,\infty)$.

\begin{theorem}
\label{thm: p Li-Yau}
    Let $p\in (1,\infty)$, $n\in \N_{\geq 2}$ and $\gamma \in W^{2,p}(\R/ \Z ; \R^n)$ be an immersed closed curve with a point of multiplicity $m \in \N_{\geq 2}$. Then 
    \begin{equation}\label{eq:p Li-Yau}
        \bar{\B}_p [\gamma] \geq \varpi_p^* m^p,
    \end{equation}
    with a constant $\varpi_p^*>0$ depending only on $p$ defined in \eqref{eq: varpi constant}.
    Moreover, equality is attained if and only if $\gamma$ is a closed $m$-leafed $p$-elastica (see Definition~\ref{def: leafed elastica}).
\end{theorem}

Theorem~\ref{thm: p Li-Yau} is used to show existence of $p$-elastic $\Theta$-networks in dimension $n\geq 2$ in Section~\ref{section: p networks}, analogous to \cite[Section~5]{miura_LiYau}.
Furthermore, we obtain an optimal energy threshold for embeddings, which directly ensures embeddedness-preserving of $p$-elastic flows, see Section~\ref{section: p-elastic flows}.

\begin{corollary}
\label{cor: embedding below}
    Let $p\in (1,\infty)$ and $\gamma \in W^{2,p}(\R / \Z; \R^n)$ be an immersed curve with normalized bending energy $\bar \B_p[\gamma] < 2^p \varpi_p^*$. 
    Then $\gamma$ is an embedding, i.e.\ it has no self-intersections.
\end{corollary}

Another motivation of this generalization is that we can observe many new phenomena concerning the optimality.
Existence of curves attaining equality, or equivalently existence of closed $m$-leafed $p$-elasticae in $\R^n$, sensitively depends on the triple $(p,m,n)\in(1,\infty)\times\N_{\geq2}\times\N_{\geq2}$.
It is trivial that if there exists a closed $m$-leafed $p$-elastica in $\R^n$, then such an elastica also exists in $\R^{n+\ell}$ for every $\ell\in\N$.
Also, recall that for every even integer $m\geq2$ there exists a closed $m$-leafed $p$-elastica in $\R^2$ (thus for all $n\geq2$); it can be achieved by an $\tfrac{m}{2}$-times covered figure-eight $p$-elastica \cite{miura_pinned_p}.
Hence the only delicate case is when $m$ is odd.

In Theorem~\ref{thm: Li-Yau equality} we obtain general existence criteria for every odd multiplicity $m\geq3$.
This completely characterizes triples $(p,m,n)$ for which equality in \eqref{eq:p Li-Yau} can be attained.
Moreover, in Theorem~\ref{thm: epsilon sharpness} we show that for any other triple there exists a threshold $\eps = \eps_{m,p}>0$ (independent of the dimension $n$) such that $\bar \B_p[\gamma] \geq \varpi_p^* m^p + \eps_{m,p}$ for any closed immersed $W^{2,p}$-curve $\gamma$ with multiplicity $m$.

Following \cite{miura_pinned_p}, by using $p_m^*$ defined in \eqref{eq:p Li-Yau}, we set
\[
p^\dagger:=p_3^*\approx 1.5728,
\]
which is the unique exponent admitting a (clover-like) closed $3$-leafed $p$-elastica, see Table~\ref{table: m-leafed elasticae}.
In the planar case, we recover \cite[Theorem 1.8]{miura_pinned_p} that full optimality holds for a unique exponent.

\begin{corollary}[{\cite[Theorem 1.8]{miura_pinned_p}}]
\label{cor: p Li-Yau equality planar}
    Let $p\in (1,\infty)$ and $n=2$.
    Then, if and only if $p=p^\dagger$, equality in \eqref{eq:p Li-Yau} can be attained for every $m\in\N_{\geq 2}$.
\end{corollary}

In addition, we can extend the above result to every dimension.
It turns out that the situation is different between $n=2$ and $n\geq3$; the above $p^\dagger$ does not extract a unique exponent but rather gives a threshold.

\begin{corollary}
\label{cor: p Li-Yau equality spatial}
    Let $p\in (1,\infty)$ and $n\in \N_{\geq 3}$.
    Then, if and only if $p\geq p^\dagger$, equality in \eqref{eq:p Li-Yau} can be attained for every $m\in\N_{\geq 2}$.
\end{corollary}

We should mention that, it is not easy to determine whether for a given exponent $p$ there exist planar $m$-leafed $p$-elasticae attaining equality in \eqref{eq:p Li-Yau}.
In \cite{miura_LiYau} it is confirmed that they do not exist for $p=2$ via Andr\'e's algebraic independence theorem.\\

\subsection{Structure of the paper}
We start in Section~\ref{section:preliminaries} by deriving initial local and global regularity from the first variation.
In Section~\ref{section:main} we prove Theorem~\ref{thm:p<=2 case}, \ref{thm:p>2 case} and \ref{thm:EL for k and tau}.
Next, we utilize these regularity results to classify pinned $p$-elastica in Section~\ref{section: pinned}. 
Section~\ref{section: applications} is devoted to the proof of Theorem~\ref{thm: p Li-Yau} and its consequences and concludes with  applications to minimal $p$-elastic networks and $p$-elastic flows.

\subsection*{Acknowledgements}
The second author is supported by JSPS KAKENHI Grant Numbers JP21H00990, JP23H00085, and JP24K00532.

\subsection*{Data availability} 
Data sharing is not applicable as we generate or analyze no data in this work.

\subsection*{Conflict of interest} 
On behalf of all authors, the corresponding author states that there is no conflict of interest.

\section{Initial regularity and local properties under curvature positivity}
\label{section:preliminaries}
By a direct computation \cite[Lemma~A.1]{miuraclassification}, Definition~\ref{def: p elastica} is equivalent to
\begin{equation}
\label{eq:first variation}
\int_0^L (1-2p) |\gamma''|^p \langle \gamma', \eta' \rangle + p |\gamma''|^{p-2} \langle \gamma'', \eta '' \rangle + \lambda \langle \gamma', \eta'\rangle = 0,
\end{equation}
for all $\eta \in C^\infty_c(0,L;\R^n)$.
By approximation, \eqref{eq:first variation} holds also for $\eta \in C^2([0,L];\R^n)$ with $\eta(0)=\eta(L)=\eta'(0)=\eta'(L)=0$.
In the case of pinned $p$-elasticae (Definition~\ref{def: pinned p elastica}), we instead consider $\eta \in C^2([0,L];\R^n)$ with $\eta(0)=\eta(L)=0$.

\begin{remark}
\label{rmk: n-1 p-elastica is n p elastica}
    Let $d<n$ be a positive integer. We note that for a $p$-elastica $\gamma \in W^{2,p}(0,L;\R^d)$, its trivial embedding $\tilde \gamma =(\gamma,0) \in W^{2,p}(0,L;\R^n)$ is also a $p$-elastica. Conversely, a  $d$-dimensional $p$-elastica $\gamma \in W^{2,p}(0,L;\R^n)$ can be seen as a $p$-elastica in $W^{2,p}(0,L;\R^d)$ by restricting $\gamma$ to its $d$-dimensional subspace.
\end{remark}

\subsection{Initial regularity}
For an arclength parameterized curve $\gamma$ we define
\begin{align*}
W&:=|\gamma''|^{p-2}\gamma'', \qquad \text{so } \gamma''= |W|^{\frac{2-p}{p-1}}W,  \\
 k&:=|\gamma''|=|\kappa|, 
  \\ \qquad w&:=|W|=k^{p-1}.
\end{align*}

We start by using the first variation \eqref{eq:first variation} to improve regularity of $k$ and $W$.

\begin{proposition}
\label{prop: initial regularity}
    Let $p\in (1,\infty)$ and $\gamma \in W^{2,p}(0,L;\R^n)$ be a $p$-elastica.
    Then
    \begin{enumerate}
        \item $\gamma \in W^{2,\infty}(0,L;\R^n)$, that is $k \in L^\infty(0,L;\R)$,
        \item $W=|\gamma''|^{p-2}\gamma'' \in W^{2,\infty}(0,L;\R^n)$.
    \end{enumerate}
\end{proposition}

\begin{proof}
    The first part follows by the same argument as in \cite[Proposition~2.1]{miuraclassification}, repeated here for completeness. 
    Fix $\phi \in C_c^\infty(0,L)$ and
    \begin{equation*}
        \xi(s):= \int_0^s \int_0^t \phi(r) dr dt + \alpha s^2 + \beta s^3, \qquad s \in [0,L],
    \end{equation*}
    with 
    \begin{equation*}
        \alpha := \frac{1}{L} \int_0^L \phi(s) ds -\frac{3}{L^2} \int_0^L\int_0^s \phi(t) dt ds, \quad         \beta := -\frac{\alpha}{L} - \frac{1}{L^3} \int_0^L\int_0^s \phi(t) dt ds.
    \end{equation*}
    It follows that $\xi \in C^{2}[0,L]$ with $\xi(0)=\xi(L)=\xi'(0)=\xi'(L)=0$, and also $|\alpha|,|\beta|, \|\xi\|_{C^1} \leq C \|\phi\|_{L^1}$ for some $C=C(L) < \infty$. 
    Taking for $i=1,\dots,n$ the test functions $\eta_i =\xi e_i$ in \eqref{eq:first variation},
    \begin{equation*}
        \left| \int_0^L p |\gamma''|^{p-2}\gamma_i'' \phi \right| \leq C' \|\phi\|_{L^1(0,L)}, \qquad \text{where }C'=C'(L,\|\gamma\|_{W^{2,p}(0,L;\R^n)}),
    \end{equation*}
    which gives $|\gamma''|^{p-2} \gamma_i''\in L^\infty(0,L)$, and hence $\gamma'' \in L^\infty(0,L;\R^n)$.

    For the second part, by \eqref{eq:first variation}, we have in the distributional sense,
    \begin{equation*}
    \label{eq: ODE for gamma in distributional sense}
        p (|\gamma''|^{p-2}\gamma'')'' - ((1-2p)|\gamma''|^p\gamma' + \lambda \gamma')' = 0.
    \end{equation*}
    Thus for some $C\in \R^n$,
    \begin{equation}
    \label{eq: distributional EL integrate1}
    p (|\gamma''|^{p-2}\gamma'')' - \underbrace{((1-2p)|\gamma''|^p\gamma' + \lambda \gamma')}_{=:g} = C.
    \end{equation}
    Then, as $g\in L^\infty(0,L;\R^n)$, we have $(|\gamma''|^{p-2}\gamma'')' \in L^\infty(0,L;\R^n)$ and thereby $|\gamma''|^{p-2}\gamma'' \in W^{1,\infty}(0,L;\R^n)$ as well as $|\gamma''|^{p-1}\in W^{1,\infty}(0,L)$. Since $f(x) = x^{\frac{p}{p-1}}$ is locally Lipschitz in $[0,\infty)$, from $\|\gamma''\|_{L^\infty(0,L;\R^n)}<\infty $ we obtain 
    $$
    |\gamma''|^p = f\circ |\gamma''|^{p-1}  \in W^{1,\infty}(0,L).
    $$ 
    This implies that $g\in W^{1,\infty}(0,L; \R^n)$, and hence $|\gamma''|^{p-2}\gamma'' \in W^{2,\infty}(0,L; \R^n)$.
\end{proof}

We obtain directly continuity of the second derivative.

\begin{corollary}
\label{cor: gamma is C^2 and a bit more}
    Let $p\in (1,\infty)$ and $\gamma \in W^{2,p}(0,L;\R^n)$ be a $p$-elastica. Then $\gamma \in C^{2,\alpha}(0,L; \R^n)$ for $\alpha= \min\{1,\tfrac{1}{p-1}\}$, and in particular $\gamma''$ and $k$ are continuous. 
\end{corollary}

\begin{proof}
    As $W$ is Lipschitz continuous and $f(x)=x|x|^{\frac{2-p}{p-1}}$ is of type $C^\alpha(\R^n;\R^n)$ with $\alpha= \min\{1,\tfrac{1}{p-1}\}$, the result follows as $\gamma''=f\circ W$.
\end{proof}

\subsection{Smoothness on positivity interval}
Since the function $W \mapsto W |W|^{\frac{2-p}{p-1}}$ for $p>2$ is not differentiable at the origin, there is no straightforward way to obtain higher regularity for $\gamma$. 
However, working locally in the set where $W\neq 0$, it is possible to improve the regularity.

\begin{definition}
    Given $p\in (1,\infty)$ and $\gamma\in W^{2,p}(0,L;\R^n)$ a $p$-elastica, we say that a relatively open interval $I\subset [0,L]$ is a \emph{positivity interval} if $|\gamma''|=k>0$ in $I$ and for each endpoint $a\in \partial I$ either $k(a)=0$ or $a\in \{0,L\}$.
\end{definition}

\begin{lemma}
\label{lemma: smoothness in positivity interval}
    Let $p \in (1,\infty)$, $\gamma \in W^{2,p}(0,L;\R^n)$ be a $p$-elastica and $I$ a positivity interval. 
    Then $\gamma|_I \in C^\infty(I;\R^n)$.
\end{lemma}

\begin{proof}
    Fix $K \subset \subset I$, so $k|_K\geq c_K > 0$. 
    Note that by \eqref{eq: distributional EL integrate1}, 
    \begin{equation}
    \label{EL_bootstrap}
        W' = \frac{1-2p}{p} |\gamma''|^p \gamma' +\frac{\lambda}{p} \gamma' +C.
    \end{equation}
    As $f:x \mapsto x|x|^{\frac{2-p}{p-1}}$ is analytic away from the origin and $\gamma'' = W |W|^{\frac{2-p}{p-1}}$, on $K$,
    \begin{equation*}
        W \in C^m(K; \R^n) \implies \gamma \in C^{m+2}(K; \R^n).
    \end{equation*}
    Initially, we have $W \in  W^{2,\infty}(K; \R^n) \subset C^1(K; \R^n)$, so $\gamma \in C^3(K; \R^n)$.
    Then the RHS in \eqref{EL_bootstrap} is also of type $C^1(K; \R^n)$ ($x\mapsto |x|^p$ is analytic away from the origin) and thus $W\in C^2(K; \R^n)$. 
    Iterating this argument gives $\gamma \in C^m(K; \R^n)$ for any $m$.
    As $K$ was arbitrary, it follows that $\gamma \in C^\infty(I; \R^n)$.
\end{proof}

In the non-degenerate regime $p\leq 2$, the same argument gives even global information. 
This is not essential for the rest of the paper, but highlights the difference to the degenerate regime.

\begin{corollary}
     Let $p\in (1,2]$ and $\gamma \in W^{2,p}(0,L;\R^n)$ be a $p$-elastica. Then $\gamma \in C^3([0,L];\R^n)$. 
\end{corollary}

\begin{proof}
The function $f: \R^n \to \R^n$ given as $f(x)=(f_1(x),\dots,f_n(x)) = |x|^{\frac{2-p}{p-1}}x$ is continuously differentiable as $|\nabla f| \leq C_{n,p} |x|^{\frac{2-p}{p-1}}$. Since $W\in W^{2,\infty}(0,L;\R^n)$,
\begin{equation*}
\gamma''' = (\gamma'') ' = (f \circ W)' = (\nabla f \circ W) \cdot W', 
\end{equation*}
with the RHS in $C([0,L];\R^n)$ and hence $\gamma \in C^3([0,L];\R^n)$.
\end{proof}

\subsection{Dimensional rigidity}

Note that \eqref{eq:first variation} is formally a fourth-order ordinary differential equation for $\gamma$, depending thus on initial conditions for $\gamma$, $\gamma'$, $\gamma''$ and $\gamma'''$.
We use this fact to show that inside a positivity interval, an a priori $n$-dimensional $p$-elastica stays three dimensional. 
Note that the initial position/zeroth-order condition is redundant; a translation does not impact the dimension of the curve $\gamma$.

\begin{proposition}
\label{prop:dimension}
    Let $p\in (1,\infty)$, $\gamma \in W^{2,p}(0,L;\R^n)$ be a $p$-elastica and $I$ a positivity interval with endpoints $a$ and $b$. 
    Then the image $\gamma(I)$ lies in an affine subspace of $\R^n$ of dimension at most three, which can be characterized by
   $$S_I= \gamma\left(\tfrac{a+b}{2} \right) + \Span\left\{\gamma'\left( \tfrac{a+b}{2}\right), \gamma''\left( \tfrac{a+b}{2}\right), \gamma'''\left( \tfrac{a+b}{2}\right)\right\}.$$ 
\end{proposition}

\begin{proof}
    Without loss of generality, we assume $\gamma\left(\frac{a+b}{2}\right) = 0$ and the interior vectors $\gamma'\left(\frac{a+b}{2}\right)$, $\gamma''\left(\frac{a+b}{2}\right)$ and $\gamma'''\left(\frac{a+b}{2}\right)$ to span $S_I:=\Span \{e_1,\dots,e_d\}$, with $d\leq 3$.
    
    Fix $\delta>0$ such that $[a+\delta,b-\delta] \subset \subset I$. 
    Then \eqref{eq:first variation} holds pointwise everywhere in $[a+\delta,b-\delta]$ by Lemma~\ref{lemma: smoothness in positivity interval}, that is
    \begin{equation}
    \label{eq:ELpointwise_complete}
    \begin{split}
        0 &=p (|\gamma''|^{p-2}\gamma'')'' - (1-2p)(|\gamma''|^p\gamma')' + \lambda \gamma''  \\
        &= p \big[ (p-2) \left( (p-4)|\gamma''|^{p-6} \langle \gamma'', \gamma''' \rangle^{2} \gamma'' + |\gamma''|^{p-4} \langle \gamma'', \gamma'''' \rangle \gamma''  \right. \\
        & \left. \qquad  + |\gamma''|^{p-4} \langle \gamma''', \gamma''' \rangle \gamma'' + |\gamma''|^{p-4} \langle \gamma'', \gamma''' \rangle \gamma''' \right) + |\gamma''|^{p-2} \gamma'''' \big] \\
        & \qquad  -(1-2p) \left(|\gamma''|^{p-2} \langle \gamma'', \gamma''' \rangle \gamma' + |\gamma''|^p \gamma'' \right) + \lambda \gamma''.
    \end{split}
    \end{equation}
    Furthermore, there are $c_\delta, M_\delta \in (0,\infty)$ such that
    \begin{equation*}
        |\gamma''|\geq c_\delta  \qquad \text{and} \qquad |\gamma'|, |\gamma''|, |\gamma'''| \leq M_\delta.
    \end{equation*}
    After moving terms of lower order, we obtain
    \begin{equation*}
         (p-2) |\gamma''|^{p-4} \langle \gamma'', \gamma'''' \rangle \gamma'' + |\gamma''|^{p-2} \gamma'''' = F(\gamma', \gamma'',\gamma'''),
    \end{equation*}
    with $F: B_{M_\delta}(0) \times \left(B_{M_\delta}(0) \setminus B_{c_\delta}(0)\right) \times B_{M_\delta}(0) \to \R^n$ Lipschitz (and analytic).
    The LHS can be rewritten as
    \begin{equation*}
    A(\gamma'')\gamma'''' \qquad \text{where} \qquad A(x):= (p-2) |x|^{p-4} x  x^T + |x|^{p-2}I.
    \end{equation*}
    For $|x|>0$, the matrix $|x|^{p-2}I$ is invertible regardless of $p$ and hence by the Sherman--Morrison formula, $A(x)$ is invertible on $B_{M_\delta}(0) \setminus B_{c_\delta}(0)$ with inverse
    \begin{equation*}
        A^{-1}(x)= |x|^{2-p} I - \frac{p-2}{p-1}|x|^{-p} x x^T.
    \end{equation*}
Therefore,
    \begin{equation}
    \label{eq:pointwise ODE for gamma4}
        \gamma''''(s)= A^{-1} (\gamma''(s)) F(\gamma'(s), \gamma''(s), \gamma'''(s)) = G(\gamma'(s), \gamma''(s), \gamma'''(s)),
    \end{equation}
    where $G: B_{M_\delta}(0) \times \left( B_{M_\delta}(0) \setminus B_{c_\delta}(0) \right) \times B_{M_\delta}(0) \to \R^n$ is also Lipschitz and analytic.

    Let $\iota:\R^d \to \R^n$ be the canonical injection and $\zeta:[a+\delta,b-\delta]\to \R^d$ the unique solution to \eqref{eq:pointwise ODE for gamma4} with interior initial values 
    \begin{equation*}
    \zeta^{(i)}\left(\tfrac{a+b}{2}\right)= \iota^{-1} \circ \gamma^{(i)}\left(\tfrac{a+b}{2}\right) \qquad \text{for } i=0,1,2,3.
    \end{equation*}
    Then $\tilde \gamma = \iota \circ \zeta$ defines a solution to \eqref{eq:pointwise ODE for gamma4} in $\R^n$ (thanks to the explicit forms of $A^{-1}$ and $F$) with the same interior initial condition as $\gamma$, which is at most three dimensional.

    By Picard--Lindelöf on \eqref{eq:pointwise ODE for gamma4}, it follows that $\gamma = \tilde{\gamma}$ and so $\gamma|_{[a+\delta, b-\delta]} \subset S_I$.
    Since $\delta$ is arbitrary and $S_I$ independent of $\delta$, the result follows.
\end{proof}

\subsection{Euler--Lagrange equation for the curvature}

Thanks to the regularity in the positivity interval, we are able to derive a strong (i.e.\ pointwise) Euler--Lagrange equation for the nonnegative-valued scalar curvature $k$.

\begin{proposition}
\label{prop: ODE}
    Let $p\in (1,\infty)$, $\gamma \in W^{2,p}(0,L;\R^n)$ be a $p$-elastica with $\lambda \in \R$ and $I$ a positivity interval. Then there exists $C=C(\gamma,I) \in \R$ such that for every $s\in I$,
    \begin{align}        
   p(k(s)^{p-1})''+  (p-1) k(s)^{p+1} - pC^2k(s)^{3-3p}- \lambda k(s) &=0,
\label{eq:strong EL for k} \\
     k(s)^{2p-2} \tau(s) &= C .\label{eq: EL for tau}
    \end{align}
\end{proposition}

\begin{proof}
    Define $T(s):=\gamma'(s)$ and
    let $\varphi \in C_c^\infty(0,L;\R^n)$ with $\spt \varphi \subset \subset I$ be arbitrary. 
    The normal $N(s):=\tfrac{\gamma''(s)}{|\gamma''(s)|}$ is well defined and differentiable on $\spt \varphi$ as $\gamma \in C^\infty(I;\R^n)$ (Lemma~\ref{lemma: smoothness in positivity interval}). Set $B(s):= \gamma'(s) \times N(s)$. 
    The torsion $\tau$ is then given as $\tau(s) = \langle N'(s), B(s) \rangle$ and is also smooth in $I$. 
   By Proposition~\ref{prop:dimension}, we have the classical Frenet--Serret equations in the at most three dimensional subspace containing $\gamma(I)$,
    \begin{equation*}
        T' = k N, \qquad N'=-kT+\tau B, \qquad B' =-\tau N.
    \end{equation*}
    We define two test functions on $[0,L]$ by
    \begin{equation*}
        \eta(s) := \begin{cases}
            \varphi(s) N(s) &\quad \text{if } s \in \spt \varphi, \\
            0 &\quad \text{otherwise},
        \end{cases} 
        \qquad \qquad
        \zeta(s) := \begin{cases}
            \varphi(s) B(s) &\quad \text{if } s \in \spt \varphi, \\
            0 &\quad \text{otherwise}.
        \end{cases} 
    \end{equation*}
    A computation from the Frenet--Serret equations yields 
    \begin{align*}
    \eta' &= \varphi' N - \varphi kT+\varphi \tau B, \\
    \eta'' &= (-2 \varphi' k- \varphi k')T + (\varphi''- \varphi k^2 - \varphi \tau^2)N + (2\varphi' \tau + \varphi \tau')B, \\
    \zeta' &= \varphi' B - \varphi \tau N, \\
    \zeta'' &= -\varphi k \tau T + (- \varphi \tau' -2\varphi' \tau )N + (\varphi'' - \varphi \tau^2)B,
    \end{align*}
    and hence
    \begin{equation*}
    \begin{split}
        \langle \gamma',\eta' \rangle &= \langle T,\eta' \rangle = -\varphi k,
       \\ \langle \gamma'',\eta'' \rangle &= k \langle N ,\eta'' \rangle = \varphi'' k - \varphi k^3 - \varphi k\tau^2,
       \\ \langle \gamma',\zeta' \rangle &= \langle T,\zeta' \rangle = 0,
       \\ \langle \gamma'',\zeta'' \rangle &= k \langle N,\zeta'' \rangle = - \varphi k \tau' - 2 \varphi' k \tau.
    \end{split}
    \end{equation*}
    Putting the expressions for $\eta$ into \eqref{eq:first variation} gives
    \begin{equation}
    \label{eq:EL1 for kappa and tau}
    \begin{split}
    0 &= \int_{\spt \varphi} (1-2p) k^p (-\varphi k) + p k^{p-2} (\varphi'' k - \varphi k^3 - \varphi k\tau^2) + \lambda (-\varphi k) \\
    &= \int_{\spt \varphi}  pk^{p-1}\varphi''+ \big( (p-1) k^{p+1} - pk^{p-1} \tau^2- \lambda k \big) \varphi.
    \end{split}
    \end{equation}
    Secondly, using $\zeta$ and integrating by parts 
    \begin{equation*}
        0=\int_{\spt \varphi} 2 k^{p-1} \tau \varphi' + k^{p-1} \tau'\varphi = \int_{\spt \varphi} (-2 (k^{p-1} \tau)' + k^{p-1} \tau' )\varphi.
    \end{equation*}
    Hence on $I$,
    \begin{equation*}
        -2(p-1) k^{p-2}k' \tau - k^{p-1} \tau' = 0,
    \end{equation*}
  and multiplying by $-k^{p-1}$, followed by integrating, yields \eqref{eq: EL for tau} for some constant $C\in \R$. 
    Since $k>0$, equivalently $\tau = \tfrac{C}{k^{2p-2}}$. 
    Upon substituting $\tau$ in \eqref{eq:EL1 for kappa and tau},
    \begin{equation*}
       0=  \int_{\spt \varphi}  pk^{p-1}\varphi''+ \big( (p-1) k^{p+1} - pC^2k^{3-3p}- \lambda k \big) \varphi,  \qquad \forall \varphi \in C_c^\infty(I).
    \end{equation*}
    After integrating the first term by parts twice, we obtain \eqref{eq:strong EL for k} by arbitrariness of $\varphi$.
\end{proof}

\begin{remark}
\label{remark: change of variable to w}
    For $p\in (1,\infty)$, a $p$-elastica $\gamma \in W^{2,p}(0,L;\R^n)$ and a positivity interval $I$, using the substitution $w:=k^{p-1}$, we obtain for $s \in I$
    \begin{equation}
    \label{eq:EL pointwise for w}
        w''(s) + \frac{p-1}{p} w(s)^{\frac{p+1}{p-1}} - C^2 w(s)^{-3} - \frac{\lambda}{p} w(s)^{\frac{1}{p-1}}=0.
    \end{equation}
    Multiplying by $2w'$ and integrating once more gives eventually
    \begin{equation}
    \label{eq:EL for w first order}
         w'(s)^2 + \frac{(p-1)^2}{p^2}  w(s)^{\frac{2p} {p-1}} - 2\lambda \frac{p-1}{p^2}w(s)^{\frac{p}{p-1}} + C^2 w(s)^{-2} =A, 
    \end{equation}
    with $A\in \R$ a constant of integration.
\end{remark}

\section{Global regularity and structure}
\label{section:main}

Given $p\in (1,\infty)$ and a $p$-elastica $\gamma \in W^{2,p}(0,L;\R^n)$, we decompose the interval $[0,L]=Y \cup Z$ where
\begin{equation*}
    Y=\{s\in[0,L]: k(s)>0\}, \qquad Z= \{s\in[0,L]: k(s)=0\}.
\end{equation*}
From the continuity of $k$ (Corollary~\ref{cor: gamma is C^2 and a bit more}), we further decompose $Y=\bigcup_{j=1}^\infty I_j$ as a countable union of disjoint relatively open (in $[0,L]$) positivity intervals. Inside each interval $I_j$, the equations \eqref{eq:EL pointwise for w} and \eqref{eq:EL for w first order} hold pointwise.
From their structure, we obtain the following lemma.

\begin{lemma}
\label{lemma: blow-up at endpoint}
    Let $p\in(1,\infty)$, $\gamma \in W^{2,p}(0,L;\R^n)$ be a $p$-elastica and $I$ a positivity interval such that $k$ vanishes at at least one endpoint. 
    Then $\gamma|_{I}$ is planar and the constant $C$ in Proposition~\ref{prop: ODE} is zero. 
\end{lemma}

\begin{proof}
 From \eqref{eq: EL for tau}, it suffices to show that $C=0$ on $I$. 
 Without loss of generality, we assume that $k$ (and thus $w$) vanishes at the left endpoint $a$ of $I$. 
 Suppose on the contrary that the constant $C$ is non-zero.
 Take a sequence $s_k\to a$. Then $w(s_k) \to 0$. 
 Consequently, the LHS in \eqref{eq:EL for w first order} diverges to $+\infty$, whereas the RHS remains bounded, which is a contradiction.
\end{proof}

We now have all the necessary tools to prove the main theorems. 
In short, we show that depending on the constant $C$ from Proposition~\ref{prop: ODE}, the $p$-elastica $\gamma$ is spatial and analytic, or planar, or a flat-core solution.

\subsection{The case $C\neq 0$} 

We show that if $C\neq 0$ on some positivity interval, then this interval extends to the whole of $[0,L]$.

\begin{proposition}
\label{prop:C>0}
Let $p\in(1,\infty)$ and $\gamma \in W^{2,p}(0,L;\R^n)$ be a $p$-elastica such that on some positivity interval $I$ the constant $C$ in Proposition~\ref{prop: ODE} is non-zero. 
Then $\gamma$ is analytic, three dimensional and $k,|\tau| \geq c$ in $[0,L]$, for some $c>0$. In particular, $Z= \varnothing$.
\end{proposition}

\begin{proof}
    By Lemma~\ref{lemma: blow-up at endpoint}, the curvature $k$ is strictly positive at the endpoints of $I$. This implies that $I=[0,L]$. By Lemma~\ref{lemma: smoothness in positivity interval} and Proposition ~\ref{prop:dimension}, the curve $\gamma$ is smooth and at most three dimensional. Moreover, Proposition~\ref{prop: ODE} gives a constant $c>0$ such that $k,|\tau| \geq c >0$ in $[0,L]$.
    In particular, the curve $\gamma$ is non-planar on $[0,L]$. 
    
    Furthermore, \eqref{eq:EL pointwise for w} holds everywhere in $[0,L]$, i.e.\ $w''(s) = f(w(s))$ with
    \begin{equation*}
        f: w \mapsto \tfrac{\lambda}{p} w^{\frac{1}{p-1}} - \tfrac{p-1}{p}w^{\frac{p+1}{p-1}} + C^2 w^{-3}, 
    \end{equation*}
    which is analytic as long as $w\neq 0$.
    Thus $w'' \in C([0,L])$ as a composition of continuous functions, i.e.\ $w \in C^2([0,L])$. 
    By a bootstrap argument, it follows that $w \in C^\infty([0,L])$. 
    As $w\geq c^{p-1}>0$, the function $w$ is even analytic by the Cauchy--Kovalevskaya theorem. 
    It then follows that $k=w^{\frac{1}{p-1}}>0$, and $\tau = \tfrac{C}{k^{2p-2}}$, as well as $\gamma$, are analytic.
\end{proof}

\subsection{The case $C=0$}

Now we examine the case where $C=0$ and $\gamma$ is partially planar. 
If the curvature vanishes at an interior point $s_0\in (0,L)$, it remains to show how the curve connects at the joint $\gamma(s_0)$. 
Note directly that if $C=0$ on one positivity interval, then $C=0$ on any other positivity interval as well, by the same blow-up argument as in the proof of Lemma~\ref{lemma: blow-up at endpoint}.
First, from the structure of the second-order ODE for $w$, we obtain higher regularity for $w$.

\begin{lemma}
\label{lemma: w is in C^2}
    Let $p\in (1,\infty)$, $\gamma \in W^{2,p}(0,L;\R^n)$ be a $p$-elastica and $I$ a positivity interval. 
    Suppose also that the constant $C$ in Proposition~\ref{prop: ODE} is zero. 
    Then $w\in C^2(\bar I)$. 
\end{lemma}

\begin{proof}
    By Corollary~\ref{cor: gamma is C^2 and a bit more} and Lemma~\ref{lemma: smoothness in positivity interval}, we have $w\in C(\bar I) \cap C^\infty(I)$, so it remains to check existence of the limits for $w'$ and $w''$ at the endpoints $a$ and $b$. 
    Since $C=0$, from \eqref{eq:EL pointwise for w}, the limit $\lim_{s\to a^+} w''(s)$ exists and from \eqref{eq:EL for w first order}, the limit $\lim_{s\to a^+} w'(s)^2$ exists as well. 
    Since $w\geq 0$, it follows that $\lim_{s\to a^+} w'(s) \geq 0$ exists. 
    An analogous argument at the endpoint $b$ gives the result. 
\end{proof}

We also have the following symmetry conditions, transferring first-order boundary conditions across the interval.

\begin{lemma}
\label{lemma: boundary condition transfers from a to b}
    Let $p\in (1,\infty)$, $\gamma \in W^{2,p}(0,L;\R^n)$ be a $p$-elastica and $I$ a positivity interval with endpoints $a$ and $b$. Suppose that $C=0$ on $I$. 
    Suppose that $w(a)=w(b)=0$ and $w'(a) = w_0>0$. 
    Then $w'(b)=-w_0$. 
    Analogously, $w'(b) = w_0<0$ implies $w'(a) = -w_0$. 
\end{lemma}

\begin{proof}
By continuity of $w$ and $w'$, and by  \eqref{eq:EL for w first order}, we obtain $A=w_0^2$, which then gives $w(b)= \pm w_0$. Since $w>0$ on $I$, the function $w$ cannot approach zero from below and thus $w(b)= - w_0$. 
The reverse case follows in the same way.
\end{proof}

Up until now we have only worked locally inside the positivity intervals.
Now we use the global initial regularity to transfer information from one positivity interval to the next.
\begin{lemma}
\label{lemma:zero point is isolated}
Let $p\in (1,\infty)$ and $\gamma\in W^{2,p}(0,L;\R^n)$ be a $p$-elastica.
    Suppose there exists $s_0 \in (0,L)$ with $w(s_0)=0$ and either the derivative from the left $w'(s_0^-)$ or the derivative from the right $w'(s_0^+)$ is non vanishing. 
    Then, both one-sided derivatives are non-vanishing at $s_0$ and have opposite sign. 
    In particular, $s_0$ is an isolated point of $Z$.
\end{lemma}
\begin{proof}
    It suffices to show that $w'(s_0^-) < 0$ implies $w'(s_0^+) > 0$; the reverse case follows by the same argument.
    
    From Proposition~\ref{prop: initial regularity}, we have $W\in W^{2,\infty}(0,L;\R^n) \subset C^1(0,L;\R^n)$ and so $|w'| \in C(0,L)$.
    Then by Cauchy--Schwarz, we obtain 
    \small
    \begin{equation*}
         0 < |w'(s_0^-)| = \lim_{s\to s_0^-}|w'(s)| = \lim_{s\to s_0^-} \left| \frac{1}{|W(s)|}\langle W(s),W'(s) \rangle \right| \leq \lim_{s\to s_0^-}|W'(s)| = |W'(s_0^-)|.
    \end{equation*}
    \normalsize
    Moreover, we have $0 \neq  W'(s_0^-) = W'(s_0) = W'(s_0^+)$. 
    Thus
    $$W(s) = W(s_0) + (s-s_0)W'(s_0) + o(|s-s_0|)= (s-s_0)W'(s_0) + o(|s-s_0|),$$
    and thereby 
    \begin{equation*}
        w'(s_0^+) = \lim_{s\to s_0^+} \frac{|W(s)|-|W(s_0)|}{s-s_0} = \lim_{s\to s_0^+} \frac{|(s-s_0)W'(s_0) + o(|s-s_0|)|}{|s-s_0|} = |W'(s_0^+)|,
    \end{equation*}
    which is strictly positive, as we wanted to show.
\end{proof}

We now show the key fact that at joints $\gamma (s)$ with $w(s)$ and $w'(s)\neq 0$, the curve $\gamma$ connects in the same plane.

\begin{proposition}
\label{prop: planar at the joint}
Let $p\in (1,\infty)$ and $\gamma \in W^{2,p}(0,L;\R^n)$ be a $p$-elastica.
    If there exists $s_0 \in (0,L)$ with $w(s_0)=0$ and $w'(s_0^-) = w_0 < 0$ (or $w'(s_0^+) = w_0 > 0$), then $\gamma$ is a wavelike planar $p$-elastica. Thus, up to  similarity, the curve $\gamma$ can be written as $(\gamma_w,0)$, where $\gamma_w \in W^{2,p}(0,L;\R^2)$ is a planar wavelike elastica from Case II in \cite[Theorem 1.2, 1.3]{miuraclassification}.
\end{proposition}

\begin{proof}
We construct suitable test functions to \eqref{eq:first variation}, which give sufficient boundary conditions at the joint. 
For notational simplicity, shift the interval of parameterization to $[-s_0,L- s_0]$ such that $s_0=0$ and assume $w'(0^-) = w_0 < 0$. 
From Lemma~\ref{lemma:zero point is isolated} it follows
that $w'(0^+) >0$. Let $s_1<0<s_2$ be sufficiently small such that $w(s)>0$ in $(s_1,0) \cup (0,s_2)$. 
It suffices to show that at $\gamma(0)$ the curve $\gamma$ connects in the same plane with $|w'(0^-)|=|w'(0^+)|$ and then apply Lemma~\ref{lemma: boundary condition transfers from a to b} to each positivity interval to obtain global planarity. 

Since $w>0$ and $C=0$ (otherwise contradicting Proposition~\ref{prop:C>0}) on $(s_1,0)$ and $(0,s_2)$, the curves $\gamma_1:=\gamma|_{[s_1,0]}$ and $\gamma_2:=\gamma|_{[0,s_2]}$ are planar. 
Let $\pi_1$ and $\pi_2$ be their supporting planes. 
Define $k_1 = k|_{(s_1,0)}$, $k_2 = k|_{(0,s_2)}$ as well as
\begin{equation*}
    \begin{split}
         N_1(s) :=  \frac{\gamma''(s)}{|\gamma''(s)|} = \frac{\gamma''(s)}{k_1(s)}  \qquad \qquad \text{on } (s_1,0), \\
        N_2(s) :=  \frac{\gamma''(s)}{|\gamma''(s)|} = \frac{\gamma''(s)}{k_2(s)}   \qquad \qquad \text{on } (0,s_2) .\\
     \end{split}
\end{equation*}
Moreover, let $T(s)=\gamma'(s)$ and $Q_1, Q_2 \in \so(n)$ be the rotations in $\pi_1$ and $\pi_2$ by an angle of $\tfrac{\pi}{2}$ satisfying $N_1(s) = Q_1 T(s)$ for $s_1<s<0$ and $N_2(s) = Q_2 T(s)$ for $0<s<s_2$. 
By the continuity of $\gamma'$, the normal vectors $N_1$ and $N_2$ have limits as $s\to 0^\pm$, which a priori may not be the same.
Set now for $s \in (s_1,s_2)$, 
$$\Tilde{N}_1(s) := Q_1 T(s), \qquad \qquad \Tilde{N}_2(s) := Q_2 T(s),$$
note that $\Tilde{N}_1|_{(s_1,0)}=N_1$ and $\Tilde{N}_2|_{(0,s_2)}=N_2$.

Take $\phi \in C_{c}^{\infty}(-1,1)$ with $\phi(0)=1$, $\|\phi'\|_{L^\infty}\leq 2$ and set $\phi_{\eps}(s)=\phi(s/\eps)$ for $\eps>0$. 
Define now for $i=1,2$ the test functions $\eta_i$ as
\begin{equation*}
         \eta_i(s) = \begin{cases}
        \int_{-\eps}^s\phi'_{\eps} \Tilde{N}_i dr - \tfrac{s+\eps}{2\eps} \int_{-\eps}^{\eps} \phi'_{\eps} \Tilde{N}_i dr \quad &\text{if } s \in (-\eps, \eps), \\
        0 \quad &\text{if } s \in (-\eps, \eps)^c.
        \end{cases}
\end{equation*}
We compute
\begin{align*}
         &\eta_i'(s) =  \chi_{(-\eps,\eps)} \phi'_{\eps}(s) \Tilde{N}_i(s) -\frac{1}{2\eps}\int_{-\eps}^{\eps} \phi'_{\eps} \Tilde{N}_i, \\
          &\eta_i''(s) = \chi_{(-\eps,\eps)} (\phi''_{\eps}(s) \tilde{N}_i(s) + \phi'_{\eps}(s) \Tilde{N}_i'(s)). 
    \end{align*}
If $\eps$ is sufficiently small, then $\eta_1$ and $\eta_2$ are in $C_c^2(-s_0,L-s_0)$ and are valid test functions in \eqref{eq:first variation}. 
Plugging $\eta_1$ into \eqref{eq:first variation} gives 
\begin{equation*}
\begin{split}
0&= \int_{-s_0}^{L-s_0} (1-2p) |\gamma''|^p \langle \gamma', \eta_1' \rangle + p |\gamma''|^{p-2} \langle \gamma'', \eta_1'' \rangle + \lambda \langle \gamma', \eta_1' \rangle \\
&= \int_{-\eps}^0 (1-2p) |\gamma''|^p \langle \gamma', \eta_1' \rangle + p |\gamma''|^{p-2} \langle \gamma'', \eta_1'' \rangle + \lambda \langle \gamma', \eta_1' \rangle \\
 &\qquad  + \int_0^{\eps} (1-2p) |\gamma''|^p \langle \gamma', \eta_1' \rangle + p |\gamma''|^{p-2} \langle \gamma'', \eta_1'' \rangle + \lambda \langle \gamma', \eta_1' \rangle  =: J_1 +J_2.\\
 \end{split}
 \end{equation*}
 Using the mean value theorem twice, first on $[-\eps,\eps]$, then on $[-\eps,0]$, we estimate
\begin{equation*}
\begin{split}
    &\int_{-\eps}^0 \frac{(1-2p) |\gamma''(s)|^p +\lambda}{2\eps} \langle T(s),\smallint_{-\eps}^{\eps} \phi'_{\eps} \tilde{N}_1 \rangle ds \\
    &\quad = \int_{-\eps}^0 \frac{((1-2p) |\gamma''(s)|^p +\lambda)2 \eps}{2\eps} \phi'_{\eps}(\xi_1) \langle T(s),\tilde{N}_1(\xi_1) \rangle ds \\
    &\quad = \eps ((1-2p) |\gamma''(\xi_2)|^p +\lambda) \phi'_{\eps}(\xi_1) \langle T(\xi_2), \tilde{N}_1(\xi_1) \rangle \\
    & \quad = \eps ((1-2p) |\gamma''(\xi_2)|^p +\lambda) \frac{1}{\eps}\phi'(\xi_1/\eps) \langle T(\xi_2), \tilde{N}_1(\xi_1) \rangle \\
    &\quad= ((1-2p) |\gamma''(\xi_2)|^p +\lambda) \phi'(\xi_1/\eps) \langle T(\xi_2), \tilde{N}_1(\xi_1) \rangle  =o(1), 
\end{split}
\end{equation*}
as $\eps \to 0$.
For $\xi_1\in [-\eps, \eps]$ and $\xi_2 \in [-\eps,0]$ we have $\langle T(\xi_2), \tilde{N}_1(\xi_1) \rangle \to 0$ as $\eps \to 0$, by continuity of $T(s)$. The other terms remain uniformly bounded, hence the bound $o(1)$ as $\eps \to 0$.
 Thus for $J_1$, we calculate
 \begin{equation}
 \label{eq: final estimate for J1}
     \begin{split}
 J_1&= \int_{-\eps}^0 (1-2p) |\gamma''|^p \langle \gamma', \eta_1' \rangle + \lambda  \langle \gamma', \eta_1' \rangle \\
 &\qquad+ \int_{-\eps}^0 p |\gamma''|^{p-2}  \langle \gamma'', \phi_{\eps}'' \Tilde{N}_1 \rangle + p |\gamma''|^{p-2} \langle \gamma'', \phi'_{\eps} \Tilde{N}_1' \rangle  \\
 &= \int_{-\eps}^0 \left( (1-2p) |\gamma''|^p+ \lambda \right) \left( \phi'_{\eps} \underbrace{ \langle T,N_1 \rangle}_{=0} - \frac{1}{2\eps} \langle T,\smallint _{-\eps}^{\eps} \phi'_{\eps} \tilde{N}_1 \rangle \right)  \\
& \qquad + \int_{-\eps}^0  p k_1^{p-1} \phi_{\eps}'' \underbrace{\langle N_1, N_1 \rangle}_{=1} + p k_1^{p-1} \phi'_{\eps} \underbrace{\langle N_1, N_1' \rangle}_{=0} \\
& = o(1)+ \int_{-\eps}^0  p k_1^{p-1} \phi_{\eps}''\qquad \text{as $\eps\to 0$}. 
\end{split}
\end{equation}
On the other hand, for $J_2$ we have
\begin{equation}
\label{eq: J2}
\begin{split}
    J_2 &= \int_0^{\eps} ((1-2p) |\gamma''|^p + \lambda)  \langle T, \phi'_{\eps} \Tilde{N}_1 \rangle +((2p-1) |\gamma''|^p - \lambda)\langle T,  \frac{1}{2\eps} \smallint_{-\eps}^{\eps} \phi'_{\eps} \tilde{N}_1  \rangle  \\
    &\qquad + p |\gamma''|^{p-2}  \langle \gamma'', \phi_{\eps}'' \Tilde{N}_1 \rangle + p |\gamma''|^{p-2} \langle \gamma'', \phi_{\eps}' \Tilde{N}_1' \rangle ds.  
\end{split}
\end{equation}
For the first integrand in \eqref{eq: J2}, again from the continuity of $T(s)$, we obtain
\begin{equation*}
   \sup_{s\in [0,\eps] } \langle T(s), \phi'_{\eps}(s) \Tilde{N}_1(s) \rangle = \sup_{s\in [0,\eps] }\langle T(s), \frac{1}{\eps} \phi'(s /\eps) Q_1 T(s) \rangle = o\left( \frac{1}{\eps} \right) \qquad \text{as $\eps\to 0$}.
\end{equation*}
For the second integrand in \eqref{eq: J2}, we again use the mean value theorem on $[-\eps,\eps]$, giving
\begin{equation*}
   \sup_{s\in [0,\eps] } \frac{1}{2\eps} \langle T(s),\smallint_{-\eps}^{\eps} \phi'_{\eps} \tilde{N}_1 \rangle =  \sup_{s\in [0,\eps] }   \frac{1}{2\eps} \langle T(s), 2 \phi'(\xi/\eps) \tilde{N}_1(\xi) \rangle = o\left( \frac{1}{\eps} \right)\qquad \text{as $\eps\to 0$}. 
\end{equation*}
Finally, for the fourth (last) integrand in \eqref{eq: J2},
\begin{equation*}
    \begin{split}
  \sup_{s\in [0,\eps] } \big|p|\gamma''(s)|^{p-2} \langle \gamma''(s), \phi'_{\eps}(s) \Tilde{N}_1'(s) \rangle \big| &\leq  \frac{p}{\eps} \sup_{(-1,1)} \phi' \sup_{s\in [0,\eps] }k_2^{p-2}(s) |\langle N_2(s),Q_1 T'(s) \rangle\big| \\
   &= \frac{2p}{\eps}\sup_{s\in [0,\eps] }k_2^{p-2}(s) |\langle N_2(s), Q_1 k_2(s) N_2(s) \rangle| \\
   &\leq \frac{2p}{\eps} \sup_{s\in [0,\eps] } k_2^{p-1}(s) = o\left( \frac{1}{\eps} \right) \qquad \text{as $\eps\to 0$}.
    \end{split}
\end{equation*}
Since we integrate over an interval of length $\eps$, 
\begin{equation}
\label{eq: final estimate for J2}
   J_2 = o(1) + \int_0^{\eps} p k_2^{p-1} \phi_{\eps}'' \langle N_2,  Q_1 T \rangle ds \qquad \text{as $\eps\to 0$}.
\end{equation}
Combining the estimates \eqref{eq: final estimate for J1} and \eqref{eq: final estimate for J2}, we have
\begin{equation}
\label{eq: final J1+J2}
    0= J_1+J_2= o(1) + \int_{-\eps}^0  w_1 \phi_{\eps}''  + \int_0^{\eps} w_2 \phi_{\eps}'' \langle N_2,  Q_1 T \rangle \qquad \text{as $\eps\to 0$},
\end{equation}
where we set $w_1:= k_1^{p-1}$ and $w_2:= k_2^{p-1}$.
By Lemma~\ref{lemma: w is in C^2}, we have  $w_1 \in C^2[s_1,0]$ and $w_2 \in C^2[0,s_2]$.
As $\gamma_1:[s_1,0]\to \R^n$ and $\gamma_2:[0,s_2]\to \R^n$ are planar $p$-elasticae, from \cite[Theorem~1.7]{miuraclassification}, it follows that $\gamma_1 \in W^{3,1}(s_1,0)$ and $\gamma_2 \in W^{3,1}(0,s_2)$.
Hence
\begin{equation*}
    |k_2'| = \big||\gamma_2''|' \big| = \left| \frac{1}{|\gamma_2''|} \langle \gamma_2'', \gamma_2''' \rangle \right| \leq |\gamma_2'''| \in L^{1}(0,s_2),
\end{equation*}
and in $(0,\eps)$ we have
\begin{equation*}
\begin{split}
    (\langle N_2,Q_1 T \rangle)' &= \langle Q_2 T' ,Q_1 T \rangle + \langle Q_2 T ,Q_1 T' \rangle \\
    &= -k_2 (\langle Q_2 N_2 ,Q_1 T \rangle + \langle Q_2 T ,Q_1 N_2 \rangle) \\
    &= -k_2 (\langle Q_2^2 T ,Q_1 T \rangle + \langle Q_2 T ,Q_1 Q_2 T \rangle =: -k_2 h(s).
\end{split}
\end{equation*}
Since $h \in W^{1,\infty}(0,s_2)$ (which follows from $\gamma \in W^{2,\infty}(0,L)$) and $k_2 \in W^{1,1}(0,s_2)$, it follows that $\langle N_2, Q_1 T\rangle' \in W^{1,1}(0,s_2)$ and hence $w_2\langle N_2,Q_1 T \rangle$ is in $W^{2,1}(0,\eps)$ as a product of a $C^2$ and a $W^{2,1}$ function.

Now we integrate \eqref{eq: final J1+J2} by parts twice.
Then most boundary terms vanish as $w_1(0)=w_2(0)=0$ and $\spt \phi_{\eps} \subset (-\eps,\eps)$, giving as $\eps\to 0$,
\begin{equation*}
\begin{split}
    0 &= o(1) + \int_{-\eps}^0  w_1 \phi_{\eps}''  + \int_0^{\eps} w_2 \phi_{\eps}'' \langle N_2,  Q_1 T \rangle \\
    &= o(1)  - \int_{-\eps}^0  w_1' \phi'_{\eps}  - \int_0^{\eps}  \phi'_{\eps} \left( w_2 \langle N_2,  Q_1 T \rangle \right)' \\
    &= o(1) -\left[w_1' \phi_{\eps} \right]_{-\eps}^0 + \int_{-\eps}^0 w_1'' \phi_{\eps} - \left[\left( w_2 \langle N_2,  Q_1 T \rangle \right)' \phi_{\eps} \right]_{0}^{\eps} + \int_{0}^{\eps} \left( w_2 \langle N_2,  Q_1 T \rangle \right)'' \phi_{\eps} \\
    &=o(1) -w_1'(0) + (w_2 \langle N_2,  Q_1 T \rangle)'(0) + \int_{-\eps}^0 w_1'' \phi_{\eps} + \int_{0}^{\eps} \left( w_2 \langle N_2,  Q_1 T \rangle \right)'' \phi_{\eps}.
\end{split}
\end{equation*}
By the dominated convergence theorem, the two integral terms vanish when passing to the limit as $\eps \to 0$. 
Thus
\begin{equation*}
\begin{split}
   0&= -w_1'(0) + w_2'(0) \langle N_2, Q_1 T \rangle(0) + w_2(0) \langle N_2, Q_1 T \rangle'(0) \\
   &= -w_1'(0) + w_2'(0) \langle N_2(0), Q_1 T(0) \rangle.
\end{split}
\end{equation*}
Since $w_1'(0^-)\neq 0$, it follows that $w_2'(0^+) \neq 0$ by Lemma~\ref{lemma:zero point is isolated}, so $w_2 \equiv 0$ in $(0,s_2)$ is impossible.
By using $\eta_2$ and analogous calculations, we arrive at
\begin{equation*}
    0=  -w_2'(0) + w_1'(0) \langle N_1(0), Q_2 T(0) \rangle.
\end{equation*}
Combining the two conditions gives
\begin{equation*}
\begin{split}
    0= w_1'(0) \left( 1-\langle N_2(0), Q_1 T(0) \rangle \langle N_1(0), Q_2 T(0) \rangle \right) \\
    \implies \langle N_2(0), Q_1 T(0) \rangle = \langle N_1(0), Q_2 T(0) \rangle = \pm 1.
\end{split}
\end{equation*}
As $w_1'(0)w_2'(0)<0$, it follows that
$$-1 = \langle N_1(0), Q_2 T(0) \rangle = \langle N_2(0),Q_1 T(0) \rangle= \langle N_1(0), N_2(0) \rangle.$$ 
This implies that $N_1(0)=-N_2(0)$ and thereby 
$$\pi_1= \textnormal{span}\{N_1(0),T(0)\}= \textnormal{span}\{N_2(0),T(0)\}=\pi_2.$$
Therefore, the curve $\gamma:[s_1,s_2] \to \R^n$, given as $\gamma = \chi_{[s_1,0]} \gamma_1 + \chi_{(0,s_2]} \gamma_2$, is a planar $p$-elastica with $w(0)=0$ and $w'(0^+)$ and $w'(0^-)$ nonzero. 
Then the condition of Case (II) in \cite[Theorem 4.1]{miuraclassification} holds, thus by \cite[Theorem 1.2, Theorem 1.3]{miuraclassification} the curve $\gamma_{[s_1,s_2]}$  is a planar wavelike $p$-elastica
contained in $\pi_1=\pi_2$, finishing the proof.
\end{proof}

We are now able to prove Theorem~\ref{thm:p<=2 case}, the regularity for $p\leq 2$. We first have a short lemma, which follows directly from the Picard--Lindel\"of theorem; the proof is therefore omitted.

\begin{lemma}
\label{lemma: uniqueness for w}
    For any $p\in(1,2]$, there does not exist a non-trivial solution $u:[a,b] \to \R$ to the Cauchy problem
\begin{equation}
    \begin{cases}
        u'' + \frac{p-1}{p} u ^{\frac{p+1}{p-1}} - \frac{\lambda}{p} u^{\frac{1}{p-1}} =0,\\
        u(a)=u'(a)=0.
    \end{cases}
\end{equation}
\end{lemma}

\begin{proof}[Proof of Theorem~\ref{thm:p<=2 case}] 
Let $\gamma$ be a non-planar $p$-elastica for $p\leq 2$.
Then there exists some positivity interval $I$ such that \eqref{eq:EL pointwise for w} holds on $I$. 
If $C\neq 0$ on $I$, then $I=[0,L]$ and the result follows by Proposition~\ref{prop:C>0} and Proposition~\ref{prop:dimension}. 

We show that the case $C=0$ cannot occur.
Clearly, if $\bar I = [0,L]$, then $\gamma$ is planar in $I$, thus suppose that $\bar I \subsetneq [0,L]$. 
Let $s_0\in \partial I \cap (0,L)$.
Then the derivative $w'(s_0)=\lim_{I\ni s\to s_0}w'(s)$ is well-defined by Lemma~\ref{lemma: w is in C^2}.
If $w'(s_0)\neq 0$, Proposition~\ref{prop: planar at the joint} implies that $\gamma$ is planar.
If $w'(s_0)=0$, Lemma~\ref{lemma: uniqueness for w} and the contrapositive to Lemma~\ref{lemma: boundary condition transfers from a to b} give $w\equiv 0$, which is a contradiction to the definition of the positivity interval.
\end{proof}

Note that in the case $p>2$ with $w$ and $w'$ vanishing on $\partial I$, Lemma~\ref{lemma: uniqueness for w} does not hold, as the Picard--Lindelöf theorem does not apply. 
We now show that there exists a non-trivial family of solutions to \eqref{eq:EL pointwise for w}, the so-called \emph{flat-core} solutions. 
Moreover, they are possibly ``proper" curves in $\R^n$, neither planar nor spatial.

In contrast to the planar setting of \cite{miuraclassification}, we recall that in the setting here the curvature $k$ takes only nonnegative values.
To simplify notation, we define 
    \begin{equation*}
        A_{p,\lambda}:= \frac{1}{2} \left( \frac{2 \lambda}{p-1} \right)^{\frac{1}{p} },
        \end{equation*}
and for $p\in (2,\infty)$,
\begin{equation*}
    T_{p,\lambda}:= \frac{K_p(1)}{A_{p,\lambda}} = \frac{1}{A_{p,\lambda}} \int_0^{\pi/2} \frac{1}{(\cos t)^{2/p}}dt,
\end{equation*}
in accordance with \cite[Section~3]{miuraclassification}.
Also the $p$-hyperbolic secant, $\sech_p$, and $p$-hyperbolic tangent, $\tanh_p$, are defined as in \cite{miuraclassification}. 
Here and hereafter $\{e_i\}_{i=1}^n$ denotes the canonical basis of $\R^n$. 
We now define the flat-core solutions.

\begin{definition}[flat-core type curvature]
\label{def: flat-core curvature}
Given $p\in (2,\infty)$, we say that $k:[0,L] \to \R_{\geq 0}$ is of \emph{flat-core} type if there exist $\lambda > 0$, an integer $N\in \N$ and $\{s_j\}^N_{j=1} \in (-T_{p,\lambda}, L + T_{p,\lambda})$ such that $s_{j+1}\geq s_j + 2 T_{p,\lambda}$ and
    \begin{equation*}
        k(s) = \sum_{j=1}^N 2A_{p,\lambda} \sech_p(A_{p,\lambda}(s-s_j)).
    \end{equation*}
\end{definition}

By definition, the positivity sets are mutually disjoint, and in fact given by
\begin{equation*}
    I_j= \{s \in [0,L]: 2A_{p,\lambda} \sech_p(A_{p,\lambda}(s-s_j)) >0\}.
\end{equation*}

We introduce the concatenation of curves for $\gamma_1:[a_1,b_1] \to \R^n$ with $L_1=b_1-a_1$ and $\gamma_2:[a_2,b_2] \to \R^n$ with $L_2=b_2-a_2$ by
\begin{equation*}
    (\gamma_1 \oplus \gamma_2)(s) := \begin{cases}
        \gamma_1(s+a_1), & s \in [0,L_1],\\
        \gamma_2(s+a_2-L_1)+ \gamma_1(b_1)-\gamma_2(a_2), & s \in [L_1,L_1+L_2].
    \end{cases}
\end{equation*}
We also define inductively $\bigoplus_{j=1}^N \gamma_j:= \gamma_1 \oplus \dots \oplus \gamma_N = (\gamma_1 \oplus \dots \oplus \gamma_{N-1})\oplus \gamma_N$.

Moreover, for $p\in (2,\infty)$, $L \geq0$ and $\theta \in \bS^{n-2} \subset \Span\{e_2,\dots,e_n\} \subset \R^n$ we let $\gamma_l^L:[0,L]\to \R^n$ and $\gamma_b^{\theta}:[-K_p(1),K_p(1)] \to \R^n$ be given by 
\begin{align*}    
   \gamma_l^{L}(s) &:= -se_1, \\
    \gamma_b^{\theta}(s) &:= (2 \tanh_p s -s)e_1+ \tfrac{p}{p-1} (\sech_p s)^{p-1}\theta.
\end{align*}

\begin{definition}
\label{def:flatcore}
 Take $p\in (2,\infty)$, $N\in \N$, lengths $L_1, \dots,L_{N+1}\geq 0$ and directions 
$$\theta_1, \dots, \theta_N \in \bS ^{n-2}\subset \Span\{e_2,\dots,e_n\}\simeq \R^{n-1}.$$
We say that $\gamma:[0,L] \to \R^n$ is a \emph{flat-core $p$-elastica} if $\gamma''\not \equiv0$ and, up to similarity, it is represented by $\gamma (s)=\gamma_f(s+s_0)$ for some $s_0\in [0,2K_p(1)+L_1)$, where the arclength parameterized curve $\gamma_f$ is defined as
\begin{equation}\label{eq:gamma_f}
    \gamma_f := \left( \bigoplus_{j=1}^N  (\gamma_l^{L_j} \oplus \gamma_b^{\theta_j}) \right).
\end{equation} 
For an example, see Figure~\ref{fig: flat-core example1}.
\end{definition}

Note that by \cite[p.~2343]{miuraclassification}, the curvature of $\gamma_f$ is of flat-core type.
We now show that flat-core $p$-elasticae are indeed  critical points.

\begin{proposition}
    Let $p\in (2,\infty)$ and $\gamma\in W^{2,\infty}(0,L;\R^n)$ be a flat-core $p$-elastica in the sense of Definition~\ref{def:flatcore}. Then $\gamma$ is a $p$-elastica in the sense of Definition~\ref{def: p elastica}.
\end{proposition}
\begin{proof}
    It suffices to show that for any $s_0 \in [0,L]$ there exists an open neighborhood of the form $U_0=(s_0-\eps, s_0+\eps)$ such that \eqref{eq:first variation} holds for any fixed $\eta \in C_c^\infty(U_0;\R^n)$.
    If $k(s_0)>0$ or $\gamma|_{U_0}$ planar, the result follows from the results for planar $p$-elastica in \cite{miuraclassification}.
Thus suppose $k(s_0)=0$. Take $\delta \in(0, \eps)$, then
    \begin{equation*}      
    \begin{split}
|I| &= \left| \int_{U_0} (1-2p) |\gamma''|^p \langle \gamma', \eta' \rangle + p |\gamma''|^{p-2} \langle \gamma'', \eta '' \rangle + \lambda \langle \gamma', \eta'\rangle \right|\\
&\leq \left|  \int_{U_0 \setminus (s_0-\delta, s_0+\delta)} (1-2p) |\gamma''|^p \langle \gamma', \eta' \rangle + p |\gamma''|^{p-2} \langle \gamma'', \eta '' \rangle + \lambda \langle \gamma', \eta'\rangle  \right| \\
& \qquad + \left|  \int_{s_0-\delta}^{s_0+\delta} (1-2p) |\gamma''|^p \langle \gamma', \eta' \rangle + p |\gamma''|^{p-2} \langle \gamma'', \eta '' \rangle + \lambda \langle \gamma', \eta'\rangle  \right|=:|I_1|+|I_2|.
\end{split}
\end{equation*}
We estimate $I_2$,
\begin{equation*}
\begin{split}
|I_2| &=\left|  \int_{s_0-\delta}^{s_0+\delta} (1-2p) |\gamma''|^p \langle \gamma', \eta' \rangle + p |\gamma''|^{p-2} \langle \gamma'', \eta '' \rangle + \lambda \langle \gamma', \eta'\rangle  \right| \\
 &\leq 2\delta C_p \|\eta\|_{W^{2,\infty}(U_0)}(\|\gamma''\|^p_{L^\infty(0,L)} \|\gamma'\|_{L^\infty(0,L)}  + \|\gamma''\|_{L^\infty(0,L)}^{p-1}  + \lambda \|\gamma''\|_{L^\infty(0,L)}) \\
&\leq C_p' \delta.
\end{split}
\end{equation*}
Note that $\sech_p$ and thereby the partially planar curve $\gamma$ are smooth on the intervals $(s_0-\eps,s_0-\delta]$ and $[s_0+\delta, s_0+\eps)$ by \cite[Proposition~3.13]{miuraclassification}. Integration by parts on  $(s_0-\eps,s_0-\delta]$ (similar for $[s_0+\delta, s_0+\eps)$) gives
\begin{equation*}
\begin{split}
   &\left| \int_{s_0-\eps}^{s_0-\delta} (1-2p) |\gamma''|^p \langle \gamma', \eta' \rangle + p |\gamma''|^{p-2} \langle \gamma'', \eta '' \rangle + \lambda \langle \gamma', \eta'\rangle \right| \\
    & \quad \leq (2p-1)\left|(\langle |\gamma''|^p\gamma',\eta \rangle\big|_{s_0-\eps}^{s_0-\delta} \right| + p \left| \langle|\gamma''|^{p-2}\gamma'',\eta' \rangle\big|_{s_0-\eps}^{s_0-\delta} \right|  +|\lambda|\left| \langle \gamma', \eta \rangle\big|_{s_0-\eps}^{s_0-\delta} \right| \\
    & \qquad +p \left| \langle(|\gamma''|^{p-2}\gamma'')',\eta \rangle\big|_{s_0-\eps}^{s_0-\delta}  \right| \\
    &\qquad  + \left| \int_{s_0-\eps}^{s_0-\delta}  \langle  - (1-2p)(|\gamma''|^p\gamma')'+ p (|\gamma''|^{p-2}\gamma'')'' + \lambda \gamma'', \eta \rangle  \right|.
\end{split}
\end{equation*}
By \cite[Theorem~1.3, Proposition~3.18]{miuraclassification} and Remark~\ref{rmk: n-1 p-elastica is n p elastica}, the curve $\gamma$ is a $p$-elastica in $[s_0-\eps,s_0-\delta]$ with strictly positive curvature, in particular smooth. Thus \eqref{eq:ELpointwise_complete} holds pointwise and the integral term vanishes. 

Note that from the explicit formula of $k(s)$ in Definition~\ref{def: flat-core curvature}, 
\begin{equation*}
  k \in C(U_0),\quad w=k^{p-1} \in C^1(U_0)    \quad \text{and} \quad k(s_0) = w(s_0)= w'(s_0)=0.
\end{equation*}
We now estimate the boundary terms.
First,
\begin{equation*}
\begin{split}
&\left|(1-2p) \langle |\gamma''|^p\gamma',\eta \rangle\big|_{s_0-\eps}^{s_0-\delta}   - p \langle|\gamma''|^{p-2}\gamma'',\eta' \rangle\big|_{s_0-\eps}^{s_0-\delta}\right| \\
&\qquad \qquad \leq C_p \|\gamma\|_{W^{2,\infty}(0,L)} \|\eta\|_{W^{1,\infty}(U_0)} (|k(s_0-\delta)|^{p-1} - |k(s_0-\delta)|^{p-1}) \\
& \qquad \qquad= o(1) \qquad \text{as }\delta \to 0.
\end{split}
\end{equation*}
Moreover, by the two dimensional Frenet--Serret frame in $[s_0-\eps,s_0-\delta]$,
\begin{equation*}
    (|\gamma''|^{p-2}\gamma'')' = (|\gamma''|^{p-1}N)' = (wN)'=w'N+wN'=w'N-wkT=w'N-k^pT,
\end{equation*}
and hence as $\delta \to 0$,
\begin{equation*}
\begin{split}
    \left| \langle(|\gamma''|^{p-2}\gamma'')',\eta \rangle\big|_{s_0-\eps}^{s_0-\delta}  \right|  \hspace{0.1cm} \leq \|\eta\|_{L^\infty(U_0)} (|w'(s_0-\delta)|+|k^p(s_0-\delta)|) =o(1).
\end{split}
\end{equation*}
It follows that $|I_1| = o(1)$ and thereby $|I| = o(1)$ as $\delta \to 0$. Thus $I=0$ and \eqref{eq:first variation} holds, finishing the proof.
\end{proof}

\begin{proposition}
\label{prop:flat-core is solution}
    Let $p\in (2,\infty)$ and $\gamma\in W^{2,p}(0,L;\R^n)$ be a $p$-elastica. 
    If some positivity interval $I$ is nonempty and $w=w'=0$ at an endpoint of $I$, then $\gamma$ is a flat-core $p$-elastica in the sense of Definition~\ref{def:flatcore}.
\end{proposition}

\begin{proof}
Let $Y=\{s\in (0,L): k(s)>0\} = \bigcup_{j} I_j$, the countable union of positivity intervals. 
By Lemma~\ref{lemma: boundary condition transfers from a to b} we have $w=w'=0$ on any $\partial I_j \subset (0,L)$.
On each $I_j=(a_j,b_j)$, the curve $\gamma$ is planar by Lemma~\ref{lemma: blow-up at endpoint}, and $k>0$. 
By arguing as in the proof of \cite[Lemma~4.8]{miuraclassification}, there exists $\bar s_j \in (a_j-T_{p,\lambda_j}, b_j+T_{p,\lambda_j})$ such that 
\begin{equation*}
    k(s)=2 A_{p,\lambda} \sech_p (A_{p,\lambda}(s-\bar s_j)),  \qquad  s \in I_j,
\end{equation*}
where $\lambda\in\R$ is the (unique) constant for which $\gamma$ satisfies \eqref{eq: first variation frechet}.
This means that $k:[0,L] \to [0,\infty)$ is of flat-core type.
Then, again arguing as in the planar case (cf.\ \cite[p.\ 2344]{miuraclassification}), and observing the elementary geometric fact that, in general codimension $n\geq2$, the loop may undergo not only reflection but also rotation with respect to the baseline (see Figure \ref{fig: flat-core example1}), we arrive at the conclusion.
\end{proof}

We are now able to prove Theorem~\ref{thm:p>2 case}, the degenerate setting.

\begin{proof}[Proof of Theorem~\ref{thm:p>2 case}] 
Let $\gamma$ be a non-planar $p$-elastica for $p> 2$. 
Then there exists some positivity interval $I$ such that \eqref{eq:EL pointwise for w} holds in $I$. If $C\neq 0$ on $I$, then $I=[0,L]$ and the result follows by Proposition~\ref{prop:C>0} and Proposition~\ref{prop:dimension}. 

If $C=0$, then there are several cases. Clearly, if $\bar I = [0,L]$, then $\gamma$ is planar, as $\tau$ vanishes in $I$. Thus suppose that $\bar I \subsetneq [0,L]$. Let $s_0 \in \partial I \cap (0,L)$.
In case $w'(s_0) \neq 0$, the curve $\gamma$ is planar by Proposition~\ref{prop: planar at the joint}.
Finally if $w'(s_0)=0$, then any non-planar $p$-elastica is a $d$-dimensional flat-core $p$-elastica by Proposition~\ref{prop:flat-core is solution}.

For the regularity, define $M_p:= \lceil \tfrac{2}{p-2} \rceil \geq 1$ and $R_p:=(M_p-\tfrac{2}{p-1})^{-1}$. We prove:
\begin{itemize}
        \item If $\tfrac{2}{p-2}$ is not an integer then  
        $\gamma \in W^{M_p+2,r}(0,L;\R^n)$ for any $r\in [1,R_p)$ but $\gamma \notin W^{M_p+2,R_p}(0,L;\R^n)$.
        \item If $\tfrac{2}{p-2}$ is an integer then $\gamma \in W^{M_p+2,\infty}(0,L;\R^n)$ but $\gamma \notin C^{M_p+2}(0,L;\R^n)$. 
    \end{itemize}
The only case to verify are two flat-core loops $\gamma_b^{\theta_1},\gamma_b^{\theta_2}:[-K_p(1),K_p(1)]\to \R^n$ connected directly, i.e.\ 
$$\gamma := \gamma_b^{\theta_1} \oplus \gamma_b^{\theta_2}:[0,4K_p(1)]\to \R^n,$$
with $\theta_1 \nparallel \theta_2$, as the other cases follow by the planar regularity results in \cite{miuraclassification}.
Since $\gamma_b^{\theta_1}$ extended by a straight line is a planar $p$-elastica, by Definition~\ref{def:flatcore} and $C^{M_p+1}$ regularity in the planar case \cite{miuraclassification}, we have $((\sech_ps)^{p-1})^{(m)}(\pm K_p(1))=0$ for any $m\leq M_p+1$, so $\gamma \in C^{M_p+1}(0,4K_p(1);\R^n)$. 
Combining this fact with the piecewise $W^{M_p+2,r}$ regularity, we have $\gamma \in W^{M_p+2,r}(0,4K_p(1);\R^n)$ for $r\in [1,R_p)$ in case $\tfrac{2}{p-2} \notin \N$ and for $r=\infty$ in case $\tfrac{2}{p-2} \in \N$.
    
    The sharpness follows from the loss of regularity of $\sech_p$, \cite[Proposition~3.13]{miuraclassification} in the case of $\tfrac{2}{p-2}$ not being an integer, analogous to \cite[Theorem~1.9]{miuraclassification}.
    If $\tfrac{2}{p-2}$ is an integer, following \cite[Eq~(B.6)]{miuraclassification}, lower order derivatives of $k$ vanish at $s_0=2K_p(1)$, but $\lim_{s\to s_0^\pm} k^{(M_p)}(s) \neq 0$. 
    Note that for the normal $N(s)$, its limits at $s_0$ from the left and right are not parallel by assumption, and hence
    \begin{equation*}
     \gamma^{(M_p+2)}(s) = (kN)^{(M_p)}(s) = k^{(M_p)}(s)N(s) + \sum_{m=0}^{M_p-1} c_m k^{(m)}(s)N^{(M_p-1-m)}(s),
    \end{equation*}
    is not continuous at $s_0=2K_p(1)$. This finishes the proof.
\end{proof}

Finally, we are able to show Theorem~\ref{thm:EL for k and tau} by using the regularity results of Theorem~\ref{thm:p<=2 case} and Theorem~\ref{thm:p>2 case}.

\begin{proof}[Proof of Theorem~\ref{thm:EL for k and tau}]
Since $\gamma$ is analytic and spatial, the curvature $k$ is strictly positive in some positivity interval $I$. 
From \eqref{eq: EL for tau} we have $k^{2p-2} \tau = C$ in $I$. If $C\neq 0$, the results follow directly from Proposition~\ref{prop:C>0}. 

We now show that the case $C=0$ leads to a contradiction with either the non-planarity or regularity assumption. 
First, if $\bar{ I} = [0,L]$, the curve is planar. Thus suppose that $\bar I \subsetneq [0,L]$ and let $s_0\in \partial I \cap (0,L)$.
If $w'(s_0) \neq 0$, then Proposition~\ref{prop: planar at the joint} implies that $\gamma$ is actually planar.
If $w'(s_0)=0$, 
then by non-planarity, necessarily $p>2$ (see Lemma~\ref{lemma: uniqueness for w}, Proof of Theorem~\ref{thm:p<=2 case}).
Thus $\gamma$ is a flat-core solution by Proposition~\ref{prop:flat-core is solution}, which is not analytic (see Theorem~\ref{thm:p>2 case}). 
\end{proof}

\section{Pinned boundary value problem}
\label{section: pinned}

In this section we show the qualitative classification of pinned $p$-elasticae.
\begin{definition}[Aligned representative]
    \label{def: flatcore representation}
    For $p\in (2,\infty)$ and a flat-core $p$-elastica $\gamma \in W^{2,p}(0,L;\R^n)$, we define its \emph{aligned representative} $\hat{\gamma}:[0,L]\to \R^n$ as follows: If 
    \[
    \gamma(s)=\Lambda A \gamma_f(\Lambda^{-1}(s+s_0))+b
    \]
    with $\Lambda \in (0,\infty)$, $A\in O(n)$, $\tfrac{s_0}{\Lambda} \in [0,2K_p(1)+L_1)$, $b\in \R^n$ and $\gamma_f$ as in \eqref{eq:gamma_f}, then we set
    \[
    \hat{\gamma}(s)=\Lambda A\hat{\gamma}_f(\Lambda^{-1}(s+s_0))+b.
    \]
    The curve $\hat{\gamma}_f$ is derived from $\gamma_f$ by aligning all flat-core loops not containing $\gamma(L)$ with the first loop (i.e.\ taking $\theta_j=\theta_1$).
    Then $\hat{\gamma}$ is at most three dimensional. In case $\gamma''(L)=0$ (i.e.\ it is the endpoint of a loop or inside a straight line segment), we also align the last flat-core loop, resulting in $\hat{\gamma}$ being two dimensional. See also Table~\ref{table: aligned flatcore} for examples.
\end{definition}

\renewcommand{\arraystretch}{1.5} 

\begin{table}[htbp]
    \centering
    \begin{tabular}{lcc} 
        \toprule
        & Original flat-core & Aligned representative \\ 
        \midrule
        $\gamma''(L) \neq 0$ & 
        \begin{minipage}{0.38\textwidth}
            \centering
            \includegraphics[width=\linewidth, trim=410 245 340 140, clip]{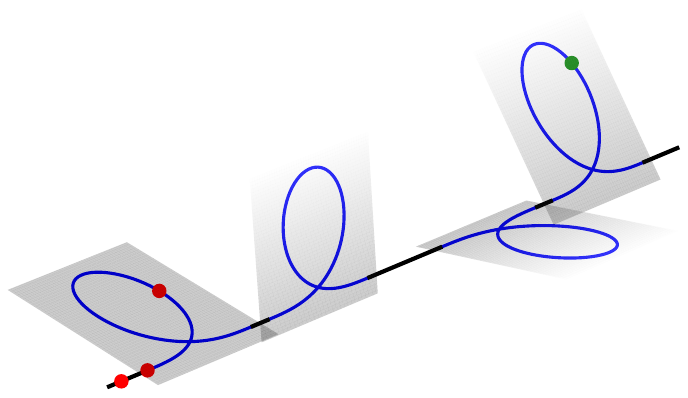}
        \end{minipage} & 
        \begin{minipage}{0.38\textwidth}
            \centering
            \includegraphics[width=\linewidth, trim=410 245 340 140, clip]{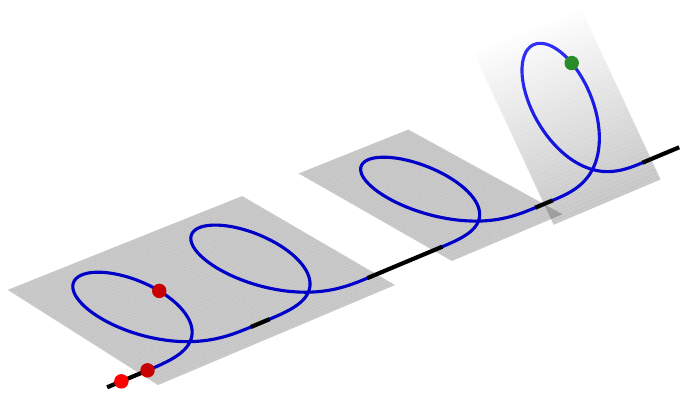}
        \end{minipage} \\
        \cmidrule(lr){1-3} 
        $\gamma''(L)=0$ & 
        \begin{minipage}{0.38\textwidth}
            \centering
            \includegraphics[width=\linewidth, trim=410 245 340 140, clip]{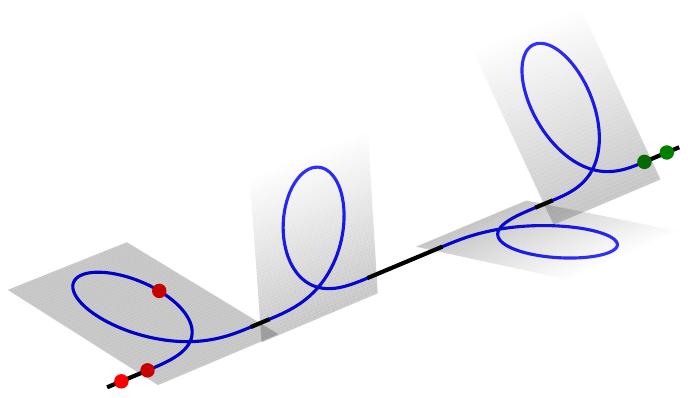}
        \end{minipage} & 
        \begin{minipage}{0.38\textwidth}
            \centering
            \includegraphics[width=\linewidth, trim=410 245 340 140, clip]{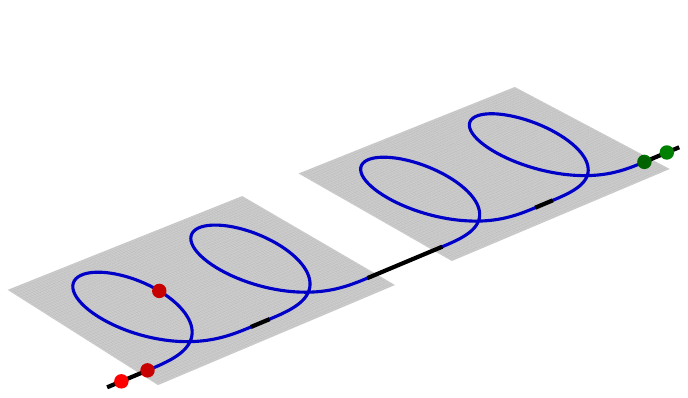}
        \end{minipage} \\
        \bottomrule
    \end{tabular}
    \caption{Aligned representative of flat-core $p$-elasticae for $p=4$. Typical examples of possible left endpoints $\gamma(0)$ are marked in red, of possible right endpoints $\gamma(L)$ in green.} 
    \label{table: aligned flatcore}
\end{table}

\begin{proof}[Proof of Theorem~\ref{thm: pinned classification}]
Let $\gamma \in \A_{P_0,P_1,L} \subset W^{2,p}(0,L;\R^n)$ be a pinned $p$-elastica.
Suppose on the contrary that $k(0)\neq 0$. 
Then by continuity (Corollary~\ref{cor: gamma is C^2 and a bit more}), $k\neq 0$ on $[0,\delta)$ for some small $\delta >0$.
Hence \eqref{eq:ELpointwise_complete} holds in $(0,\delta)$ pointwise. 
By testing \eqref{eq:first variation} against $\eta \in C^\infty_0(0,L;\R^n)$ with $\spt \eta \subset [0,\delta]$ and first-order conditions $\eta'(0) = -\gamma''(0)$ and $\eta'(\delta) = 0$ we obtain,
\begin{equation*}
\begin{split}
    0 &= \int_0^\delta (1-2p) |\gamma''|^p \langle \gamma', \eta' \rangle + p |\gamma''|^{p-2} \langle \gamma'', \eta '' \rangle + \lambda \langle \gamma', \eta'\rangle \\
    &= \int_0^\delta \langle  p (|\gamma''|^{p-2}\gamma'')'' - (1-2p)(|\gamma''|^p\gamma')' + \lambda \gamma'', \eta \rangle ds + \big[p|\gamma''|^{p-2} \langle \gamma'', \eta'\rangle \big]_0^\delta\\
    &=  p|\gamma''(0)|^{p}.
\end{split}
\end{equation*}
The integral term vanishes due to $\gamma$ being a $p$-elastica.
Hence we deduce $k(0)=0$, contrary to the assumption. Analogously, we have $k(L)=0$.

The only possibility for $\gamma$ to be non-planar is $\gamma$ being a non-trivial flat-core $p$-elastica, otherwise contradicting Theorems~\ref{thm:p<=2 case}, \ref{thm:p>2 case} and \ref{thm:EL for k and tau}.
For $p\leq2$, flat-core $p$-elasticae do not exist, and thus $\gamma$ is planar. Let $p>2$ and suppose that $\gamma$ is a flat-core $p$-elastica.
Then its aligned representative $\hat \gamma$ is a two dimensional pinned $p$-elastica in the sense of \cite[Definition~3.2]{miura_pinned_p}.
By \cite[Theorem~1.1]{miura_pinned_p}, necessarily $|P_0-P_1|\geq \tfrac{L}{p-1}$. 
\end{proof}

\begin{proof}[Proof of Corollary~\ref{cor: unique pinned minimizer}]
    The existence of a minimizer $\gamma$ follows from a standard direct method argument as in \cite[Proposition~4.1]{miura_pinned_p}. For uniqueness, thanks to \cite[Theorem~1.4]{miura_pinned_p}, we only need to prove the planarity of a minimizer $\gamma$. If otherwise ($\gamma$ being non-planar), then its aligned representative $\hat{\gamma}$ would give a planar flat-core of the same energy, which would also be a (planar) minimizer, but this contradicts \cite[Theorem~1.4]{miura_pinned_p}.
\end{proof}

\section{Applications}
\label{section: applications}

We utilize the previously derived regularity and structure results to generalize a Li--Yau type inequality. 
We also characterize the solution curves that lead to equality and discuss applications to minimal $p$-elastic networks and $p$-elastic flows.

\subsection{Li--Yau inequality}
For $p\in (1,\infty)$ we define  $q_p^* \in (0,1)$ as the unique real solution of $2\tfrac{E_{1,p}(q)}{K_{1,p}(q)}=1$, where $E_{1,p}$ and $K_{1,p}$ are defined as in \cite[Definition~3.2]{miuraclassification}.
Moreover, define the constant (as in \cite[Equation~(5.1)]{miura_pinned_p})
\begin{equation}
\label{eq: varpi constant}
    \varpi_p^*:= 2^{3p-1} (q_p^*)^{p-2} (2 (q_p^*)^2-1)E_{1,p}(q_p^*)^p.
\end{equation}

\begin{definition}
\label{def: n fold figure-eight}
Let $p\in (1,\infty)$. 
A curve $\gamma$ is called an \emph{$\tfrac{N}{2}$-fold figure-eight $p$-elastica}, if, up to similarity, it is given by a wavelike (see \cite{miuraclassification}) $p$-elastica with parameter $q_p^*$, i.e.\ for $s \in [0,2N K_{1,p}(q_p^*)]$, 
\begin{equation*}
  \gamma(s) =
    (2 E_{1,p}(\am_{1,p}(s,q_p^*),q_p^*)-s) e_1   
    -q_p^* \tfrac{p}{p-1}|\cn_p(s,q_p^*)|^{p-2} \cn_p(s,q_p^*) e_2.
\end{equation*}
By \cite[Proposition~4.3]{miura_pinned_p}, the normalized $p$-bending energy of a $\tfrac{1}{2}$-fold figure-eight $p$-elastica is given by $\varpi_p^*$.
\end{definition}

\begin{definition}
\label{def: crossing angle}
    For $p\in (1,\infty)$, we define the \emph{crossing angle}
    \begin{equation*}
        \phi^*(p):= \pi -2 \arcsin(q^*_p).
    \end{equation*}
    Note that by \cite[Proposition 4.6 (iii)]{miura_pinned_p}, the value $2\phi^*(p)$ is the angle between the tangent vectors at the two endpoints of a $\tfrac{1}{2}$-fold figure-eight $p$-elastica. 
    Moreover, by \cite[Theorem~1.5]{miura_pinned_p}, the function $ p\mapsto \phi^*(p)$ from $(1,\infty)$ to $(0,\pi/2)$ is continuous, surjective and strictly decreasing.
\end{definition}

We first show a key building block of Theorem~\ref{thm: p Li-Yau}, essentially the ``multiplicity one'' case.

\begin{proposition}
\label{prop: one leaf Li-Yau}
    Let $p\in (1,\infty)$ and $\gamma \in W^{2,p}(0,1;\R^n)$ be an immersed curve such that $\gamma(0)=\gamma(1)$. 
    Then
    \begin{equation*}
        \bar \B_p[\gamma] \geq \varpi_p^*,
    \end{equation*}
    with equality if and only if $\gamma$ is a $\tfrac{1}{2}$-fold figure-eight $p$-elastica.
\end{proposition}

\begin{proof}
Up to reparameterization and scale invariance, we may assume $\gamma$ to be arclength parameterized and $\Ll[\gamma]=1$.
Then, by Corollary~\ref{cor: unique pinned minimizer} with $P_0=P_1$, there exists a unique minimizer, which is planar.
Hence the assertion follows by the known planar result, \cite[Corollary 5.1]{miura_pinned_p}.
\end{proof}

We now define a special class of curves, characteristic of the equality case in Theorem~\ref{thm: p Li-Yau}.

\begin{definition}
\label{def: leafed elastica}
For $p\in (1,\infty)$, we call an immersed $W^{2,p}$-curve $\gamma:[0,1]\to \R^n$ an \emph{(open) $m$-leafed $p$-elastica} if there are $0 = a_1 < a_2 < \dots < a_{m+1} = 1$ such that for each $i = 1,\dots, m$ the curve $\gamma_i := \gamma|_{[a_i ,a_{i+1}]}$ is a $\tfrac{1}{2}$-fold figure-eight $p$-elastica and $\Ll[\gamma_1] = \dots = \Ll[\gamma_m]$. 
The point $\gamma(a_1) = \dots =\gamma(a_{m+1})$ is called the \emph{joint}.

Similarly, we call a closed immersed $W^{2,p}$-curve $\gamma:\R / \Z \to \R^n$ a \emph{closed $m$-leafed $p$-elastica} if there is $t_0 \in \R / \Z$ such that the curve $\tilde \gamma: [0, 1] \to \R^n$ defined by $\tilde \gamma(t) := \gamma (t + t_0 )$ is an open $m$-leafed $p$-elastica.
\end{definition}

 \begin{remark}
     We note that any closed $m$-leafed $p$-elastica $\gamma$ is of class $C^2(\R /\Z;\R^n)$. 
     The continuity of the first derivative follows directly from the classical embedding $W^{2,p}(\R /\Z;\R^n) \hookrightarrow C^1(\R /\Z;\R^n)$ and since $k$ vanishes at the joint by definition, the second derivative is continuous as well. 
     Moreover, its normalized bending energy is calculated as $\bar \B_p[\gamma]=\varpi_p^* m^p$.
 \end{remark}

\begin{remark}
    Note also that a closed $m$-leafed $p$-elastica does not necessarily have $m$-fold rotational symmetry, since it is always possible to concatenate an $m'$-leafed $p$-elastica with $\tfrac{m-m'}{2}$ figure-eight $p$-elasticae, see Figure~\ref{fig:5impureplot}.
\end{remark}

\begin{figure}[htbp]
        \centering
        \includegraphics[width=60mm]{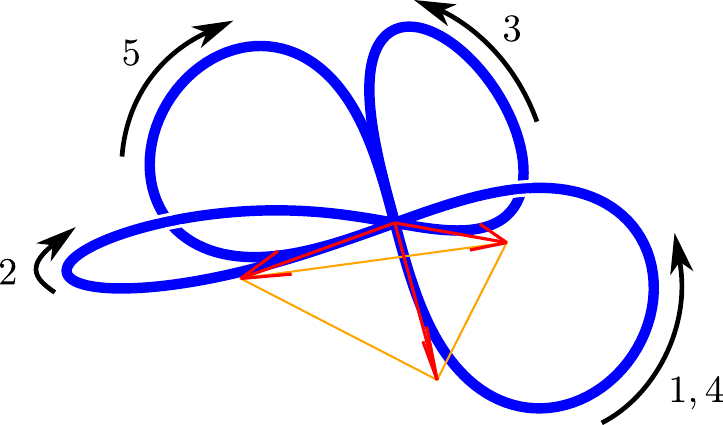}
        \caption{A $5$-leafed elastica constructed from a $3$-leafed elastica concatenated with a figure-eight elastica. Starting from the point of multiplicity $m=5$ at the origin, the loops are traversed in their numerical order.}
        \label{fig:5impureplot}
\end{figure}

Closed leafed $p$-elasticae induce a first-order condition of the corresponding open leafed $p$-elasticae at the endpoints. 
Namely the tangent vectors at the endpoints have to agree, otherwise violating the $C^1$ continuity. 
This leads naturally to conditions on the crossing angle $\phi^*(p)$, and we have the following equivalence of $m$-leafed $p$-elasticae (see also \cite[Lemma~3.7]{miura_LiYau}, \cite[Lemma~5.7]{miura_pinned_p}). 

\begin{lemma}
\label{lemma: characterization tuple}
\textnormal{(Characterization of closed $m$-leafed $p$-elasticae)}
Let $m\geq 1$ be an integer and $p\in (1,\infty)$. 
Let $\Omega^*(m,p,n)$ be the set of all $m$-tuples $(\omega_1,\dots, \omega_m)$ of $n$-dimensional unit-vectors $\omega_1,\dots,\omega_m$ in $\R^n$, such that $\langle \omega_{i+1},\omega_i \rangle = \cos 2 \phi^*(p) $ for any $i=1,\dots, m$, where we interpret $\omega_{m+1}=\omega_1$.
Then for any $m$-tuple $(\omega_1,\dots, \omega_m)\in \Omega^*(m,p,n)$, there exists a unique arclength parameterized closed $m$-leafed $p$-elastica $\gamma \in W^{2,p}(\R / \Z; \R^n)$ such that $\gamma(0)=0$ and $\gamma'\left( \tfrac{i}{m} \right)=\omega_i$ for $i=1,\dots ,m$. Conversely, for any closed $m$-leafed $p$-elastica $\gamma \in W^{2,p}(\R / \Z ;\R^n)$, there exists an $m$-tuple $(\omega_1,\dots, \omega_m)\in \Omega^*(m,p,n)$ where $\omega_i$ is given by $\gamma'\left( \tfrac{i}{m} + t_0\right)$ for some $t_0 \in \R/\Z$.
\end{lemma}

We are now able to prove Theorem~\ref{thm: p Li-Yau} by a simple partition argument.

\begin{proof}[Proof of Theorem~\ref{thm: p Li-Yau}]
    Without loss of generality, assume that the point of multiplicity $m$ is the origin. 
    Cut the curve $\gamma \in W^{2,p}(\R/\Z;\R^n)$ at the point of multiplicity and define an immersed open curve $\tilde \gamma:[0,1] \to \R^n$ with $ \tilde \gamma(0)= \tilde \gamma(1)=0$. 
    Since $\tilde \gamma$ has multiplicity $m+1$, there exist $0=a_1 < a_2 < \dots < a_m<a_{m+1}=0$ such that $\tilde \gamma(a_i)=0$ for any $i=1,\dots,m+1$. 
    Set $\tilde \gamma_i = \tilde \gamma|_{[a_i,a_{i+1}]}$ for $i=1,\dots,m$, from Proposition~\ref{prop: one leaf Li-Yau} we obtain 
    \begin{equation*}
     \Ll_p[\tilde \gamma_i]^{p-1}   \B_p[\tilde \gamma_i]  = \bar \B_p[\tilde \gamma_i] \geq \varpi_p^*.
    \end{equation*}
    It then follows that
    \begin{equation}
    \label{eq:Li-Yau sum inequality}
        \bar \B_p[\gamma]=\bar \B_p[\tilde \gamma] = \left( \sum_{i=1}^m \Ll[\tilde\gamma_i] \right)^{p-1} \sum_{i=1}^m \B_p[\tilde \gamma_i] \geq \varpi_p^* \left( \sum_{i=1}^m \Ll[\tilde \gamma_i] \right)^{p-1} \sum_{i=1}^m \left( \frac{1}{\Ll[\tilde \gamma_i]} \right)^{p-1}.
    \end{equation}
    If $p>2$, we use Jensen's inequality for the second sum on the RHS, whereas if $p\leq 2$, the reverse of Jensen's inequality for the first sum and then in each case the HM-AM inequality (parallel to \cite[Theorem 5.2]{miura_pinned_p}).
    Thus we arrive at
    \begin{equation*}
        \begin{split}
        \bar \B_p[\gamma] \geq \varpi_p^* m^p,
        \end{split}
    \end{equation*}
    as we wanted to show.

Now we discuss the equality case.
    Let $\gamma$ be a closed $m$-leafed $p$-elastica in the sense of Definition~\ref{def: leafed elastica}, where each leaf $\gamma_i$ has length $L$. 
    Then by Proposition~\ref{prop: one leaf Li-Yau},
    \begin{equation}
        \B_p [\gamma] = m \B_p[\gamma_i] = m L^{1-p} \varpi_p^*,
    \end{equation}
    and hence $\bar \B_p[\gamma] = (mL)^{p-1} \B_p[\gamma] = m^p \varpi_p^*$. 
    Conversely, suppose that $\gamma$ attains equality in \eqref{eq:p Li-Yau}, and split $\gamma$ at its point of multiplicity into open curves $\tilde \gamma_i:[a_i,a_{i+1}] \to \R^n$ with $\tilde \gamma(a_i)= \tilde \gamma(a_{i+1})$ for $i=1,\dots,m$. 
    Then each term in \eqref{eq:Li-Yau sum inequality} attains equality, i.e.\ $\B_p[\tilde \gamma_i]=\varpi_p^* \Ll[\tilde \gamma_i]^{1-p}$ and $\tilde\gamma_i$ is thereby given by a $\tfrac{1}{2}$-fold figure-eight $p$-elastica thanks to Proposition~\ref{prop: one leaf Li-Yau}. 
    The equalities $\Ll[\tilde \gamma_1]=\dots = \Ll[\tilde \gamma_m]$ follow from equality in the HM-AM and Jensen's inequality. 
\end{proof}

From the proof we have the immediate consequence.
\begin{corollary}
\label{cor: p Li-Yau for open curves}
    Let $p\in (1,\infty)$ and $\gamma \in W^{2,p}(0,1;\R^n)$ be an immersed curve with  multiplicity $m+1 \in \N_{\geq 3}$ and $\gamma(0)=\gamma(1)$. Then $\bar\B_p[\gamma] \geq \varpi_p^* m^p$ with equality if and only if $\gamma$ is an open $m$-leafed $p$-elastica.
\end{corollary}

Finally, we characterize the triples $(p,m,n)$ attaining equality in \eqref{eq:p Li-Yau}. See Table~\ref{table: m-leafed elasticae} for a selection of such curves.

\begin{theorem}
\label{thm: Li-Yau equality}
For any odd integer $m\geq 3$, there exists a finite set $P_m\subset(1,\infty)$ with the following properties: Let $p_m^*=\min P_m\in(1,\infty)$.
\begin{enumerate}
    \item If $p\in (1,p_m^*)$, then for any $n\geq 2$ there does not exist a closed $m$-leafed $p$-elastica in $\R^n$.
    \item If $p\in P_m$, then there exists a closed $m$-leafed $p$-elastica in $\R^2$.
    \item If $p\in(p_m^*,\infty)\setminus P_m$, then there exists a closed $m$-leafed $p$-elastica in $\R^3$, but not in $\R^2$.
\end{enumerate}
In addition, we have $|P_m|=\tfrac{(m-1)(m+1)}{8}$, the strict inclusion $P_m\subsetneq P_{m+2}$, strict monotonicity $p_{m+2}^*<p_m^*$ and also $p_m^*\to1$ as $m\to\infty$.
\end{theorem}

\begin{proof}
Let $m$ be fixed and consider the set 
\[
P_m = \left\{ p_{i,m'} := (\phi^*)^{-1}\left(\tfrac{i \pi}{2 m'}\right) : \; 3 \leq m' \leq m, \; m' \in 2\Z + 1, \\
\; 1 < i < m', \; i \in 2\Z \right\},
\]
where the crossing angle $\phi^*$ is given by Definition~\ref{def: crossing angle}. 
(For explicit numerical values of $P_m$, see Figure~\ref{fig:Pm_plot}.)

\begin{figure}[htbp]
    \centering
        \centering
    \includegraphics[width=0.8\textwidth]{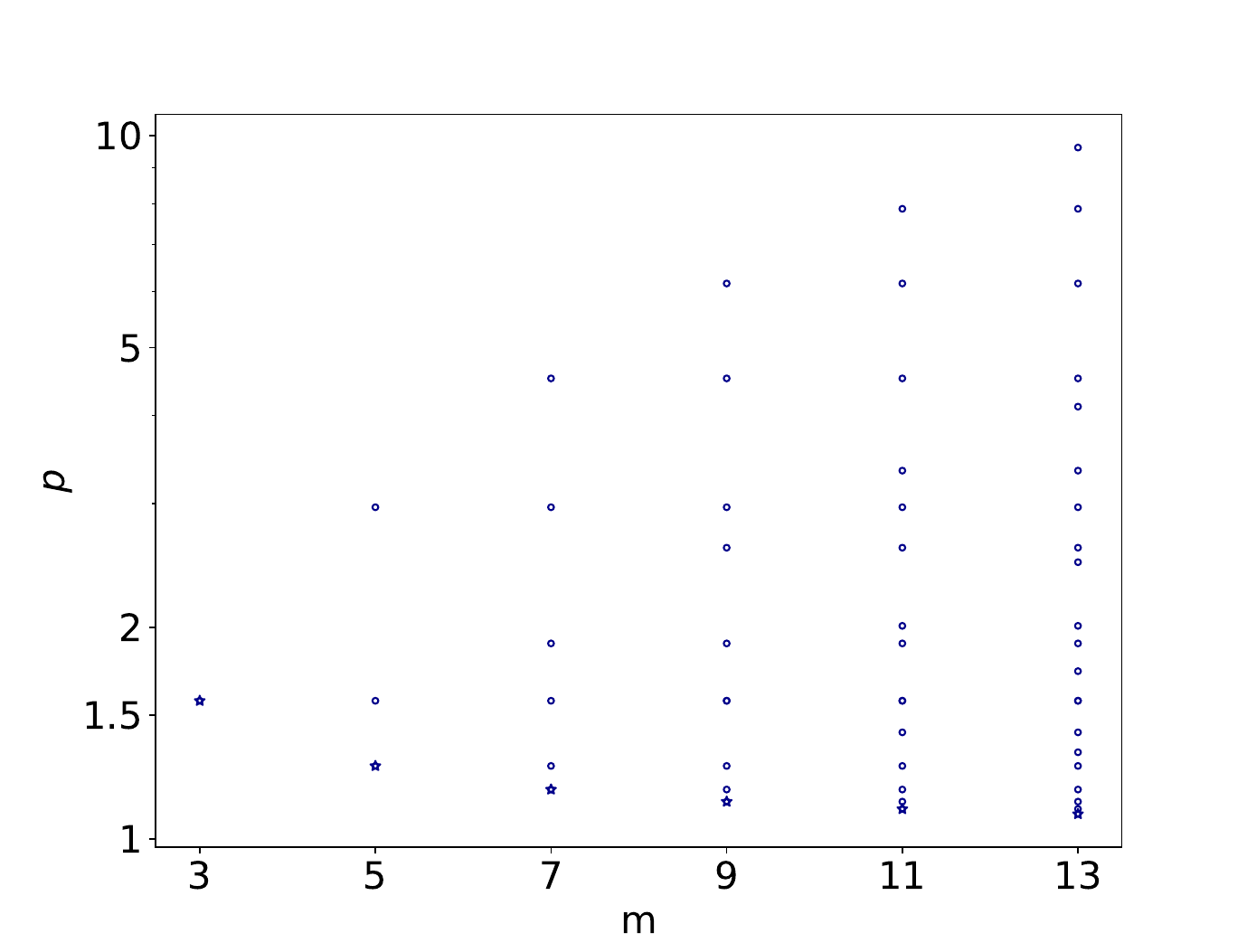}
        \caption{The sets $P_m$ for different $m$, the values $p_m^*$ are emphasized.}
        \label{fig:Pm_plot}
\end{figure}

Directly by definition, $p^*_m=\min P_m = (\phi^*)^{-1}\left(\tfrac{m-1}{2m} \pi\right)$. Moreover, from the monotonicity of the crossing angle $\phi^*$ (Definition~\ref{def: crossing angle}, \cite[Theorem~1.5]{miura_pinned_p}) we have $|P_m|=\tfrac{(m-1)(m+1)}{8}$ and $P_m \subsetneq P_{m+2}$ as well as $p^*_{m+2}<p_m^*$ and  $p_m^*\to 1$.

\textbf{Case $p<p^*_m$:} 
We argue by contradiction. By Lemma~\ref{lemma: characterization tuple}, suppose that there exists $(\omega_1, \dots, \omega_{m})\in \Omega^*(m,p,n)$ for some $n$.
Since $p < (\phi^*)^{-1}(\tfrac{m-1}{2m}\pi)$, and $\phi^*$ strictly decreasing, we have $\pi - \pi/m < 2\phi^*(p) < \pi$. 

Without loss of generality, suppose that $\omega_1$ (the unit tangent at the first endpoint of the first leaf) is given by $e_n$, the North Pole of $\bS^{n-1}$. 
Then $\omega_2$ is contained in the geodesic open ball (in $\bS^{n-1}$) of radius $\tfrac{\pi}{m}$  centered at the South Pole.
Then $\omega_3$ is contained in the geodesic open ball of radius $\tfrac{2\pi}{m}$ centered at the North Pole. 
By repeating this procedure $m$ times (particularly odd times), we have $\omega_{m+1}=\omega_1$ being contained in the open geodesic ball of radius $\pi$ centered at the South Pole, meaning that it cannot be $e_n$. 
Therefore, any $m$-tuple of tangents $(\omega_1, \dots, \omega_{m})$ cannot be closed up and $\Omega^*(m,p,n)$ is actually empty for any $n$.

\textbf{Case $p\in P_m$:} 
By definition of $P_m$, we can write $\phi^* = \phi^*(p)= \tfrac{j}{2m'} \pi$ for some odd positive integer $m'\leq m$ and even integer $j<m'$. Pick $\omega_1 \in \bS^1 \subset \R^2$. 
Denote by $R_\phi$ the counterclockwise rotation in $\R^2$ by an angle $\phi$. 
We set inductively $\omega_{i+1}:=R_{2 \phi^*} \omega_{i}$ for $i=1,\dots,m'-1$.
Then $$\omega_{m'}= R_{2 \phi^*} \omega_{m'-1} = \dots = R_{2(m'-1) \phi* } \omega_1 =  R_{-2 \phi^*} \omega_1$$ 
and thereby $\langle \omega_{i+1}, \omega_i \rangle=\cos(2\phi^*(p))$ for every $i$ up to $m'$, that is $(\omega_1,\dots, \omega_{m'}) \in \Omega^*(m',p,2)$. 
Thus Lemma~\ref{lemma: characterization tuple} gives existence of a closed $m'$-leafed $p$-elastica in $\R^2$. 
Note that, if $m'<m$, i.e.\ $m=m'+2\ell$ with $\ell\in \N$, we construct a closed $m$-leafed $p$-elastica by concatenating a closed $m'$-leafed $p$-elastica with an $\ell$-fold figure-eight $p$-elastica ($2\ell$ leaves).

\textbf{Case $p\in (p^*_m, \infty)\setminus P_m$:} 
For existence, thanks to Lemma~\ref{lemma: characterization tuple} it suffices to find $\omega_1,\dots,\omega_m \in \bS^2 \subset \R^3$ such that $\langle \omega_{i+1}, \omega_i \rangle= \cos (2 \phi^*(p))$. Let
\begin{equation*}
h= \sqrt{1-\frac{\sin^2(\phi^*(p))}{\sin^2(\frac{m-1}{2m} \pi)}} \in (0,1). 
\end{equation*}
This is well defined since $0<2\phi^*(p)<2\phi^*(p_m^*)=\tfrac{m-1}{m}\pi$.
Consider now the circle $C_h \subset \bS^2$ at latitude $-h$ i.e.\ $C_h=\{(x_1,x_2,x_3) \in \bS^2: x_3=-h \}$ with center-point $c_h=(0,0,-h)$. 
Pick $\omega_1, \dots,\omega_m$ to be equidistributed points in $C_h$, enumerated such that (see Figure~\ref{fig: idea of proof}) 
$$
\angle (\omega_{i+1}-c_h,\omega_i-c_h) = \frac{m-1}{m} \pi.
$$
Again, we interpret $\omega_{m+1}=\omega_1$. 
Then
\begin{equation*}
\tfrac{1}{2}|\omega_{i+1}-\omega_i|=\sin\left(\tfrac{m-1}{2m}\pi\right)\sqrt{1-h^2},
\end{equation*}
and hence
\begin{equation*}
    \angle ( \omega_{i+1}, \omega_i ) =  2\arcsin\left( \sin\left(\tfrac{m-1}{2m}\pi\right)\sqrt{1-h^2} \right) = 2 \phi^*(p),
\end{equation*}
as was to be shown. 

\begin{figure}[htbp]
    \centering
    \includegraphics[width=0.8\linewidth, trim=50 130 50 100, clip]{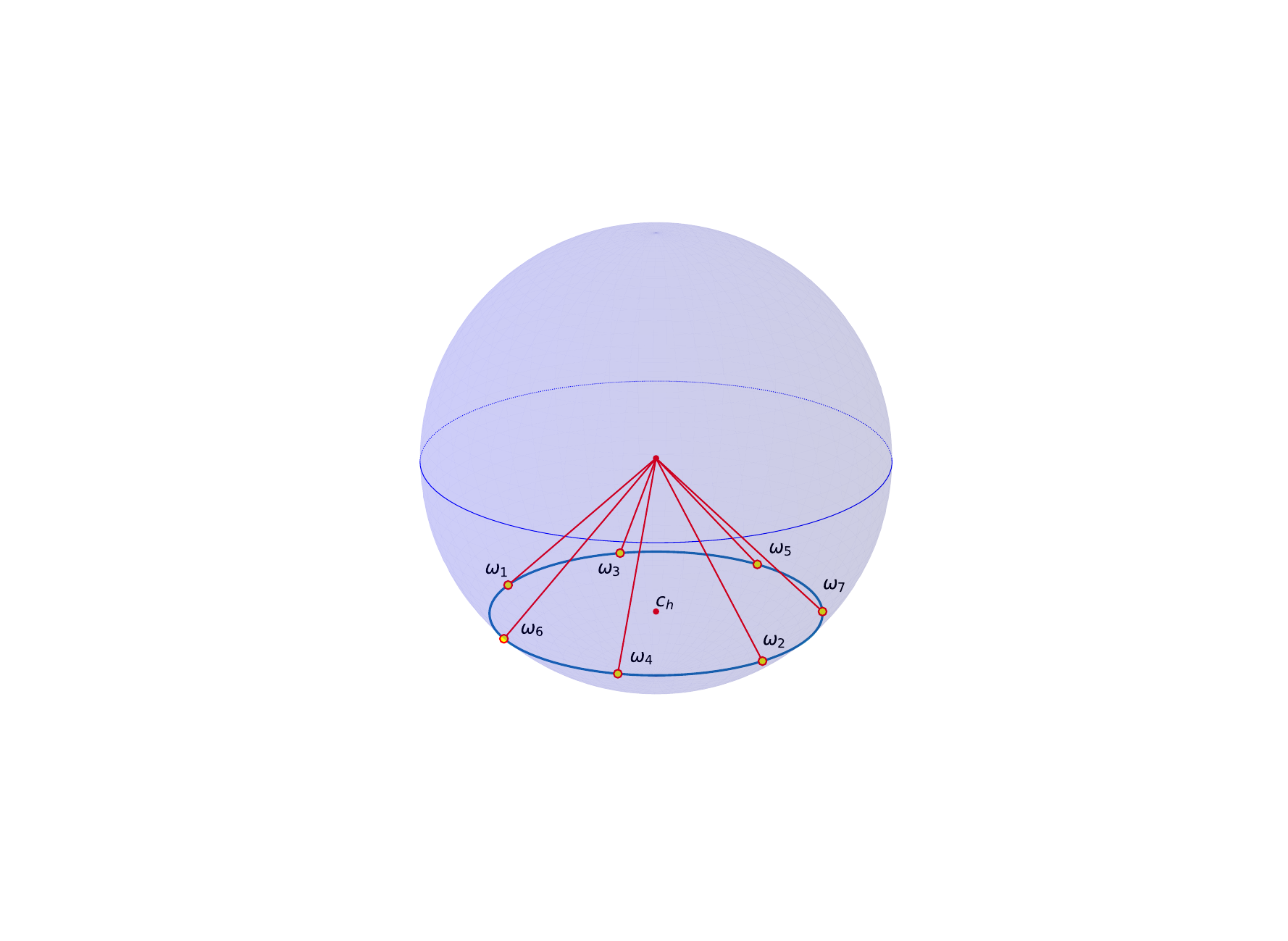} 
    \caption{Construction in the case $m=7$ and $p\in (p_m^*,\infty)\setminus P_m$}
    \label{fig: idea of proof}
\end{figure}

For non-existence, suppose by contradiction that there exists a closed arc-length parameterized $m$-leafed $p$-elastica $\gamma$ in $\R^2$. 
Since $\phi^*(p)< \tfrac{\pi}{2}$, there exist $\{\sigma_i\}_{i=1}^m$ with $\sigma_i \in \{-1,1\}$, and $m' \in 2\Z+1$ with $1\leq m' \leq m$, such that the angle sum is of the form
\begin{equation*}
    T:= \left| \sum_{i=1}^{m} \angle \left(\gamma'\left(\tfrac{i-1}{m}\right),\gamma'\left(\tfrac{i}{m}\right) \right) \right| = \left| \sum_{i=1}^{m} 2 \sigma_i \phi^*(p) \right| = m' 2 \phi^*(p) <m' \pi.
\end{equation*}
Since $m \in 2\Z+1$, necessarily $T\geq 2\phi^*(p) >0$.
However, as $\gamma$ is a closed curve in $W^{2,p}$, i.e.\ $\gamma'\left(0\right) = \gamma'\left(1\right)$, we need to have $T=\ell \pi$ for some $\ell \in 2\Z$ with $2 \leq \ell<m'$. 
This in turn gives $m'\geq 3$ and $\phi^*(p) = \tfrac{\ell}{2m'}\pi$, that is by definition $p \in P_m$ (as $m'\leq m$), which is a contradiction to $p\in (p_m^*,\infty) \setminus P_m$.
\end{proof}

\renewcommand{\arraystretch}{1.5} 

\begin{table}[htbp]
    \centering
    \begin{tabular}{lcc}
        \toprule
        & $m=3$ & $m=5$ \\
        \midrule
        $
        \begin{aligned}
             p&= p_5^*  \\
             &\approx 1.270 
        \end{aligned}$ &
        \begin{minipage}{0.38\linewidth}
            \centering
            $m$-leafed $p$-elasticae do not exist
        \end{minipage} &
        \begin{minipage}{0.38\linewidth}
            \centering
            \includegraphics[width=\linewidth, trim=80 105 80 80, clip]{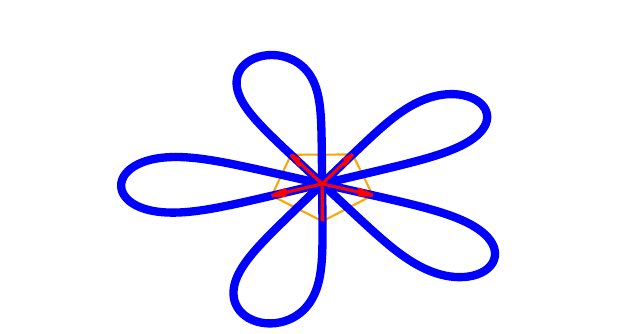}
        \end{minipage} \\
        \midrule
        
        $
        \begin{aligned}
             p&= p^\dagger = p_3^*\\
             &\approx 1.573
        \end{aligned}$ &
        \begin{minipage}{0.38\linewidth}
            \centering
            \includegraphics[width=\linewidth, trim=120 160 110 158, clip]{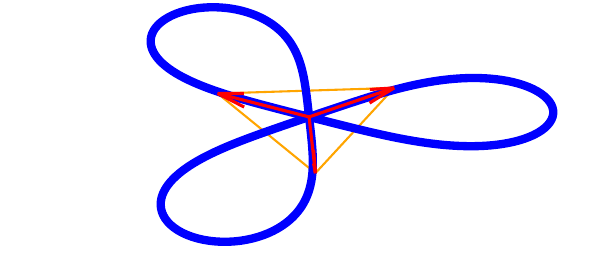}
        \end{minipage} &
        \begin{minipage}{0.38\linewidth}
            \centering
            \includegraphics[width=\linewidth, trim=100 110 100 100, clip]{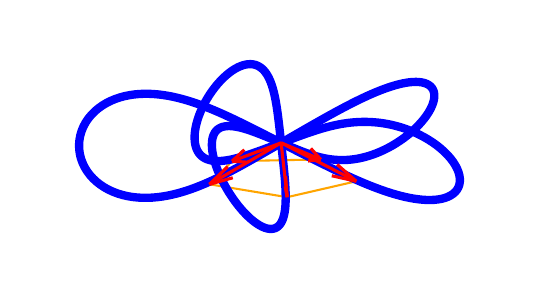}
        \end{minipage} \\
        \midrule
        
        $p=2$ & 
        \begin{minipage}{0.38\linewidth}
            \centering
            \includegraphics[width=\linewidth, trim=200 170 190 160, clip]{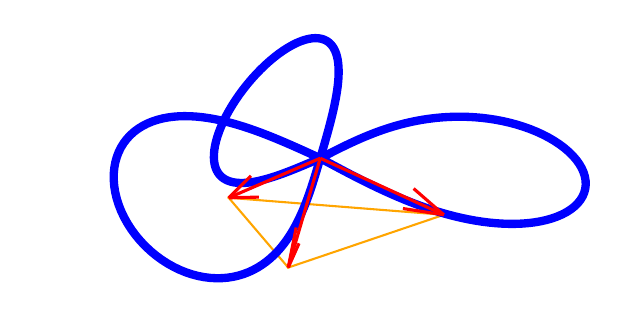}
        \end{minipage} &
        \begin{minipage}{0.38\linewidth}
            \centering
            \includegraphics[width=\linewidth, trim=60 100 60 90, clip]{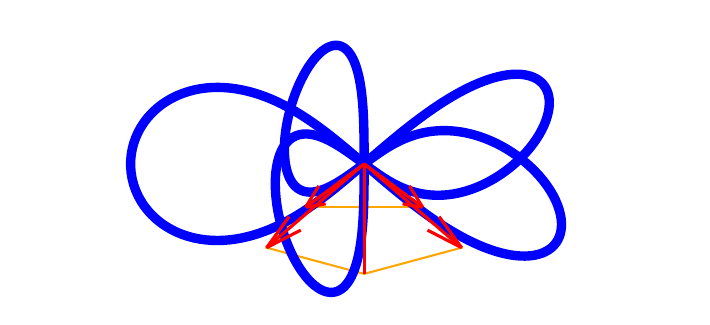}
        \end{minipage} \\
        \midrule
        
        $p=5$ &
        \begin{minipage}{0.38\linewidth}
            \centering
            \includegraphics[width=\linewidth, trim=170 130 160 130, clip]{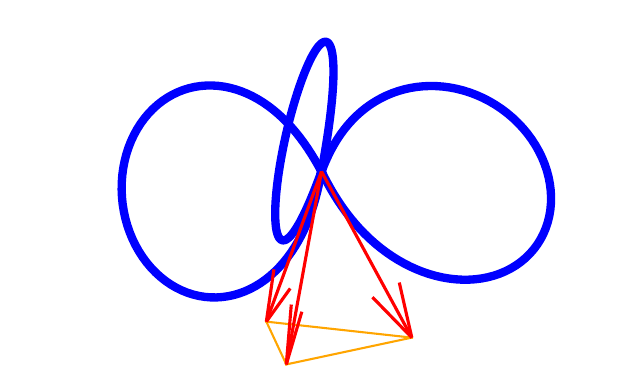}
        \end{minipage} &
        \begin{minipage}{0.38\linewidth}
            \centering
            \includegraphics[width=\linewidth, trim=70 90 70 100, clip]{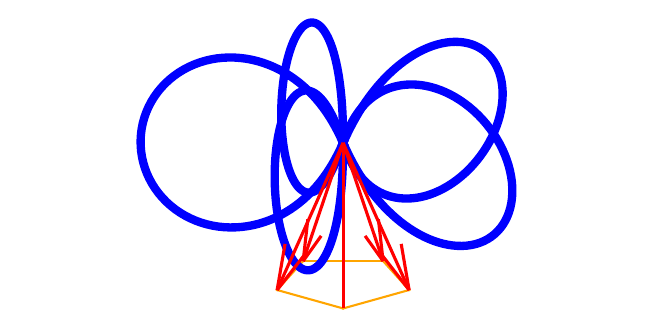}
        \end{minipage} \\
        \bottomrule
    \end{tabular}
    \caption{A selection of $m$-leafed $p$-elasticae, where each leaf is traversed only once. } 
    \label{table: m-leafed elasticae}
\end{table}

In case where equality cannot be attained, we extract energy thresholds $\eps_{m,p}$, depending only on the exponent $p$ and the multiplicity $m$, giving non-optimality of \eqref{eq:p Li-Yau}. 
Thus, in the following result we have a generalization of \cite[Theorem~1.3]{miura_LiYau}.
\begin{theorem}
\label{thm: epsilon sharpness}
    Let $m \geq 3$ be an odd integer, and let $P_m$ and $p_m^*$ be as defined in Theorem~\ref{thm: Li-Yau equality}. Then we have two cases:
    \begin{enumerate}
    \item If $p\in (1,p_m^*)$, then there exists $\eps_{m,p}>0$ such that for any $n\in \N_{\geq 2}$ and any closed immersed $W^{2,p}$-curve $\gamma: \R /\Z \to \R^n$ with a point of multiplicity $m$,
        \begin{equation*}
            \bar\B_p[\gamma] \geq \varpi_p^* m^p + \eps_{m,p}.
        \end{equation*}
    \item If $p \in (p_m^*,\infty) \setminus P_m$, then there exists $\eps_{m,p}>0$ such that for any closed immersed $W^{2,p}$-curve $\gamma: \R /\Z \to \R^2$ with a point of multiplicity $m$,
        \begin{equation*}
            \bar\B_p[\gamma] \geq \varpi_p^* m^p + \eps_{m,p}.
        \end{equation*}
    \end{enumerate}
\end{theorem}
\begin{proof}
   For $n\geq 2$, let
    \begin{equation*}
    \begin{split}
    \C_m^n&:=\{\gamma \in W^{2,p}(\R/\Z;\R^n): \text{$\gamma$ is immersed and has a point of multiplicity }m\}\\
    \end{split}
    \end{equation*}
    and
    \begin{equation*}
        \beta_m^n := \inf_{\gamma \in \C_m^n} \bar\B_p[\gamma].
    \end{equation*}
    It suffices to show that for $p<p_m^*$, or $p\geq p_m^*$ with $p \notin P_m$, we have $\beta_m^n \geq \varpi_p^* m^p + \eps_{m,p}$ with $\eps_{m,p}$ independent of $n$. 
  
  \textbf{Step 1:} We show that there exists a closed curve $\bar \gamma^n \in \C_m^n$ such that the infimum is attained.
  Without loss of generality, we assume the minimizing sequence $\gamma_j \in \C_m^n$ of $\bar\B_p$ to be arclength parameterized (in particular, of unit length) and the point of multiplicity $m$ to be the origin. That is for each $j$ we have $a_j^1 <\dots<a_j^i<\dots < a_j^m$ such that $\gamma_j(a_j^i)=0$ for any $i=1,\dots, m$.
  By the direct method (see \cite[Proposition~3.17]{miura_LiYau}), there exists an arclength parameterized minimizer $\bar \gamma^n \in W^{2,p}(\R/\Z; \R^n)$, i.e.\ $\beta_m^n=\bar \B_p[\bar \gamma^n] = \B_p[\bar \gamma^n]$ and $\gamma_j \to \bar \gamma^n$ weakly in $W^{2,p}(\R/\Z;\R^n)$ and strongly in $C^1$. 
  
  We now show that adjacent points $a_j^i$ and $a_j^{i+1}$ (we again identify $a_j^{m+1}=a_j^1$) do not collide as $j\to \infty$, so $\bar \gamma^n$ still has mulitplicity $m$ (that is $\bar \gamma^n \in \C_m^n$). By Hölder's inequality for the total curvature $\int_\gamma k ds$ and \cite[Lemma~5.2]{miura_LiYau} we have
  \begin{equation*}
     \pi^p\leq  \left(\int_{a_j^i}^{a_j^{i+1}} k ds \right)^p \leq |a_j^{i+1}-a_j^i|^{p-1} \cdot \int_{a_j^i}^{a_j^{i+1}} k^p ds  \leq
     C_m |a_j^{i+1}-a_j^i|^{p-1}.
  \end{equation*}
  For the last inequality, we used 
  $$\B_p[\bar \gamma^n|_{a_j^i}^{a_j^{i+1}}] \leq\B_p[\bar \gamma^n] = \beta_m^n \leq \beta_m^2 =:C_m.$$
  Thus $|a_j^{i+1}-a_j^i|\geq\delta_{m,p}>0$ for any $i$ and $j$. 
  Hence no points collide and up to a subsequence the $a_j^i$'s converge to $m$ distinct points $a^i$ in $\R/\Z$ as $j\to \infty$. 
  By the $C^1$ convergence, we have $\bar \gamma^n (a^i)=0$ for every $i=1,\dots,m$, and hence $\bar \gamma^n \in \C_m^n$.

 \textbf{Step 2:} The case $p\geq p_m^*$ and $p\notin P_m$ follows directly from a contradiction argument. Suppose that $\beta_m^2 = \varpi_p^* m^p$. Then by Theorem~\ref{thm: p Li-Yau} it follows that $\bar \gamma^2 \in W^{2,p}(\R / \Z; \R^2)$ is a planar $m$-leafed $p$-elastica. This contradicts the nonexistence result of Theorem~\ref{thm: Li-Yau equality}. 
 
 \textbf{Step 3:} The case $p<p_m^*$ requires an argument independent of the ambient dimension $n$. 
 First, we recall that $\pi - \tfrac{\pi}{m} < 2 \phi^*(p) < \pi$. 
 Take now $m$ and $n$ arbitrary and consider $\bar \gamma^n \in \C_m^n$ with $\bar\B_p[\bar \gamma^n]=\beta_m^n$. 
 Set $\phi_i := \tfrac{1}{2}\angle((\bar \gamma^n) '(a^i), (\bar \gamma^n)'(a^{i+1}))$. Since $\bar \gamma^n$ is closed up to first order, necessarily $\sum_{i=1}^m 2\phi_i = m' \pi$ for an even integer $m'\leq m-1$. Thus
\begin{equation*}
    \sum_{i=1}^m 2\phi_i \leq (m-1) \pi < 2 m \phi^*(p).
\end{equation*}
Therefore, there exists some $i$ such that $\phi_i\leq \tfrac{m-1}{2m} \pi < \phi^*(p)$; by a relabeling, we assume $i=1$. 
From the inequalities $2\phi_1 \leq \pi-\tfrac{\pi}{m} < 2 \phi^*(p)$, we obtain an explicit lower bound of the form $\phi^*(p)-\phi_1 \geq \phi^*(p)-\tfrac{1}{2}(\pi-\tfrac{\pi}{m})=:\eta_{m,p}>0$, independent of $n$. 

For $\phi_1 \in [0,2\pi)$, let $\bar \gamma_{\phi_1}^n$ be a minimizer (existence again follows by a direct method argument and strong $C^1$ convergence ensures that the limiting curve is still admissible, see the proof of \cite[Proposition~4.1]{miura_pinned_p}) of $\bar \B_p$ in the class
\begin{equation*}
\begin{split}
   \C^n_{\phi_1} = \{\gamma \in W^{2,p}(0,1;\R^n) \,:\, &|\gamma'|\equiv 1, \quad \gamma(0)=\gamma(1)=0, \\
   &\angle(\gamma'(0),e_1) = \angle(\gamma'(1),e_1) = \phi_1\}.
\end{split}
\end{equation*}

Since $\bar \gamma_{\phi_1}^n$ is a $p$-elastica, by Theorem~\ref{thm:p<=2 case} and \ref{thm:p>2 case}, it is either analytic and three dimensional or a flat-core $p$-elastica. In the latter case, its aligned representative (Definition~\ref{def: flatcore representation}) has the exact same $p$-bending energy. Hence, we may assume that $\bar \gamma_{\phi_1}^n$ is at most three dimensional. From Proposition~\ref{prop: one leaf Li-Yau} and the fact that $\phi^*(p)-\phi_1\geq \eta_{m,p}$, the curve $\bar \gamma_{\phi_1}^n$ cannot be a $\tfrac{1}{2}$-fold figure-eight $p$-elastica and so
$$\bar\B_p[\bar \gamma_{\phi_1}^n]-\varpi_p^* = \bar\B_p[\bar \gamma^3_{\phi_1}]-\varpi_p^* \geq  \eps_{m,p}>0. $$

Now we go back to the estimate of $\bar \gamma^n$.
By optimality, we have 
$$\B_p[\bar \gamma^n|_{[a^1,a^2]}] \Ll^{p-1}[\bar \gamma^n|_{[a^1,a^2]}] = \bar\B_p[\bar \gamma^n|_{[a^1,a^2]}]\geq \bar \B_p[\bar \gamma^n_{\phi_1}]\geq \varpi_p^* +  \eps_{m,p}.$$ 
Thus we estimate,
\begin{equation*}
\begin{split}
    \bar \B_p[\bar \gamma^n] &= \left( \B_p[\bar \gamma^n|_{[a^1,a^2]}]+ \sum_{i=2}^m\B_p[\bar \gamma^n|_{[a^i,a^{i+1}]}] \right) \Ll^{p-1}[\bar \gamma^n]\\
    &\geq \left( \frac{\varpi_p^*+\eps_{m,p}}{\Ll^{p-1}[\bar \gamma^n|_{[a^1,a^2]}]} + \sum_{i=2}^m\frac{\varpi_p^*}{\Ll^{p-1}[\bar \gamma^n|_{[a^i,a^{i+1}]}]}  \right) \Ll^{p-1}[\bar \gamma^n] \\
    &= \frac{\eps_{m,p}}{\Ll^{p-1}[ \bar \gamma^n|_{[a^1,a^2]}]} \Ll^{p-1}[\bar \gamma^n]  + \varpi_p^* \sum_{i=1}^{m} \frac{1}{\Ll^{p-1}[\bar \gamma^n|_{[a^i,a^{i+1}]}]} \Ll^{p-1}[\bar \gamma^n] \\
    &\geq \varpi_p^* m^p + \eps_{m,p},
\end{split}
\end{equation*}
using the HM-AM and Jensen's inequality as in the proof of Theorem~\ref{thm: p Li-Yau}. Hence we have for any $n \geq 2$ and $\gamma \in W^{2,p}(\R/\Z; \R^n)$ the inequality 
$$\bar \B_p[\gamma] \geq \inf_{2\leq n' \leq n} \B_p [\bar \gamma^{n'}] \geq \varpi^*_p m^p + \eps_{m,p},$$
which finishes the proof. 
\end{proof}

\subsection{Existence of minimal $p$-elastic networks}
\label{section: p networks}
One important application of the Li--Yau type inequality \eqref{eq:p Li-Yau} is the existence of minimal $p$-elastic networks. 
The most prominent setting here is the minimization of the bending energy of so-called $\Theta$-networks. For $\alpha \in (0,\pi)$ given, a triplet $\Gamma=(\gamma_1,\gamma_2, \gamma_3) \in W^{2,p}(0,1;\R^n)^3$ is called a $\Theta$-network with angles $(\alpha, \alpha, 2\pi-2\alpha)$ if 
\begin{align*}
&\gamma_1(0)=\gamma_2(0)=\gamma_3(0),\\
&\gamma_1(1)=\gamma_2(1)=\gamma_3(1),\\
&\angle (\gamma_1'(0),\gamma_2'(0)) = \angle (\gamma_1'(1),\gamma_2'(1)) = \angle (\gamma_2'(0),\gamma_3'(0))=\angle (\gamma_2'(1),\gamma_3'(1)) = \alpha, \\
&\angle (\gamma_1'(0),\gamma_3'(0)) = \angle (\gamma_1'(1),\gamma_3'(1))=2\pi-2\alpha.    
\end{align*}
The set of all such triplets is denoted by $\Theta(p,\alpha)$. We set
\begin{equation*}
    \bar\B_p[\Gamma] = \Ll[\Gamma]^{p-1} \B_p[\Gamma] := \left( \sum_{i=1}^3 \Ll[\gamma_i] \right)^{p-1} \left( \sum_{i=1}^3 \B_p[\gamma_i] \right).
\end{equation*}

However, establishing the existence of minimizers in the class $\Theta(p,\alpha)$ requires subtle arguments to avoid collapsing of the $\Theta$-network (or parts of it) to a point \cite{dallacqua_minimal_elastic_networks}.  
Using the Li--Yau type inequality for $p=2$, previous work \cite{miura_LiYau} extends the existence result of \cite{dallacqua_minimal_elastic_networks} for $p=n=2$ to $p=2$ and $n\geq2$. In a different direction, the work \cite{miura_pinned_p} extends it to the setting $n=2$ and $p\in (1, \infty)$. 

We note that \cite[Lemma~6.1, 6.2, 6.4]{miura_pinned_p} are independent of the ambient dimension $n$, and using Corollary~\ref{cor: p Li-Yau for open curves} for general $n$ in \cite[Lemma~6.3]{miura_pinned_p} instead of \cite[Theorem~5.2]{miura_pinned_p}, the existence result \cite[Theorem~1.9]{miura_pinned_p} extends to $n\geq2$ and $p \in (1,\infty)$. 

\begin{theorem}
    Let $p \in (1,\infty)$ and $\alpha \in (0,\pi)$ such that $0< \alpha < \pi-\phi^*(p)$. Then there exists  $\bar \Gamma \in \Theta(p,\alpha)$ such that $\bar \B_p[\bar \Gamma] = \inf_{\Gamma \in \Theta(p,\alpha)} \bar \B_p[\Gamma]$.
\end{theorem}

\subsection{Embeddedness of $p$-elastic flows}
\label{section: p-elastic flows}

Consider a dynamic flow of closed curves $\gamma: [0,T) \times \R /\Z \to \R^n$ with initial condition $\gamma(0,\cdot)= \gamma_0$. 
Two examples within the framework of this paper are (i) the length-preserving and (ii) the length-penalized $p$-elastic flow.
Here we define those flows in an abstract way.

\begin{enumerate}
    \item We call $\gamma$ a \emph{length-preserving $p$-elastic flow} if $\Ll[\gamma(t,\cdot)]\equiv \Ll[\gamma_0]$ holds for $t\geq0$, while $\B_p[\gamma(t,\cdot)]<\B_p[\gamma_0]$ holds for $t>0$ unless $\gamma(t,\cdot)\equiv\gamma_0$ (stationary).
    \item Given a constant $\lambda>0$, we call $\gamma$ a \emph{length-penalized $p$-elastic flow} if $(\B_p+\lambda\Ll)[\gamma(t,\cdot)]<(\B_p+\lambda\Ll)[\gamma_0]$ holds for $t>0$ unless stationary.
\end{enumerate}

For concrete examples of equations for those flows and appropriate local-in-time existence results, see e.g.\ \cites{Okabe_p_flow,Okabe_p_flow2} for the planar case and \cites{p-elastic_flow_blatt,p-flow_blatt_minimizing_movement} for general dimension, and the references therein.

Assuming the existence of a length-preserving/penalized $p$-elastic flow, a natural question to ask is:
 
\emph{Suppose that the initial curve $\gamma_0$ is embedded. Then, what is the maximal energy threshold below for which the $p$-elastic flow remains embedded for all times $t\geq 0$?} 

For classical elastic flows with $p=2$, the energy threshold $C^*=4\varpi_2^* \approx 112.439$ (the normalized bending energy of a $1$-fold figure-eight elastica) and its optimality for $n \geq 3$ have been obtained by the Li--Yau type inequality from \cite{miura_LiYau} and the construction of explicit perturbations in \cite{miura_mueller_rupp}.

For general $p\in(1,\infty)$, thanks to Corollary~\ref{cor: embedding below}, we can deduce embeddedness-preserving results analogous to \cite{miura_mueller_rupp}.

\begin{theorem}
    Let $p\in (1,\infty)$, $n\geq2$ and $\gamma:[0,T)\times\R/\Z\to\R^n$ be a length-preserving $p$-elastic flow with an embedded initial datum $\gamma_0\in W^{2,p}(\R/\Z;\R^n)$.
    Suppose that
    $$\bar \B_p[\gamma_0] \leq 2^p \varpi_p^*.$$
    Then $\gamma(t,\cdot)$ remains embedded for all $t\in[0,T)$.
\end{theorem}

\begin{proof}
    If $\gamma$ is stationary, then embeddedness is clearly preserved.
    In the non-stationary case, we deduce by definition that $\bar \B_p[\gamma(t,\cdot)]<\bar \B_p[\gamma_0]\leq 2^p \varpi_p^*$ holds for $t>0$, so by Corollary \ref{cor: embedding below} each $\gamma(t,\cdot)$ is embedded.
\end{proof}

\begin{theorem}
    Let $p\in (1,\infty)$, $n\geq2$ and $\gamma:[0,T)\times\R/\Z\to\R^n$ be a length-penalized $p$-elastic flow for $\lambda>0$ with an embedded initial datum $\gamma_0\in W^{2,p}(\R/\Z;\R^n)$.
    Suppose that
    $$\B_p[\gamma_0] + \lambda \Ll[\gamma_0] \leq 2p\left( \frac{\lambda}{p-1}\right)^{\frac{p-1}{p}} (\varpi_p^*)^{\frac{1}{p}}.$$
    Then $\gamma(t,\cdot)$ remains embedded for all $t\in[0,T)$.
\end{theorem}

\begin{proof}
    From Young's inequality, $ab \leq \frac{1}{p}a^p + \frac{1}{p'}b^{p'}$ with $\frac{1}{p}+\frac{1}{p'}=1$, by choosing $a=\B_p[\gamma_0]^\frac{1}{p}$ and $b= (\lambda \Ll[\gamma_0])^{\frac{p-1}{p}} (p-1)^{\frac{p}{p-1}}$, we obtain
    \begin{equation*}
        \begin{split}
          \bar \B_p[\gamma_0] = \B_p[\gamma_0] \Ll[\gamma_0]^{p-1} \leq \left( \frac{p-1}{\lambda}\right)^{p-1} \frac{1}{p^p} (\B_p[\gamma_0] + \lambda \Ll[\gamma_0])^p \leq 2^p \varpi_p^*.
        \end{split}    
    \end{equation*}
    Thus the proof is reduced to the length-preserving case.
\end{proof}

Analogous to the results in \cite{miura_mueller_rupp}, we conjecture that the above criteria yield the optimal thresholds in dimension $n \geq 3$ for many concrete $p$-elastic flows studied in the literature. 
Establishing this rigorously would require suitable results on the continuous dependence of solutions on initial data, which, to the authors' knowledge, are currently only available for $p = 2$.

In addition, for dimension $n=2$, we expect a higher threshold similar to the energy $\bar{\B}_2$ of the two-teardrop curve $\gamma_{2T}$ in \cite{miura_mueller_rupp}.

\bibliography{refs}

\end{document}